\documentclass[journal]{IEEEtran}
\usepackage[utf8]{inputenc}
\usepackage{mathtools}
\usepackage{amssymb}
\usepackage{bbold}
\usepackage[yyyymmdd,hhmmss]{datetime}
\usepackage{bm,upgreek}
\usepackage{color}
\mathtoolsset{showonlyrefs,showmanualtags}
\usepackage{dsfont}
\usepackage{booktabs}
\usepackage[final]{graphicx}
\usepackage{nth}
\usepackage{tikz}
\usetikzlibrary{quotes,angles}

\usepackage{cite}

\DeclareMathOperator{\Binopdf}{binopdf}
\newcommand{\Ninf}[1]{ \| #1 \| }
\DeclareMathOperator{\Conv}{Conv}
\newcommand{\Exp}[1]{\mathbb{E}[ #1]}
\newcommand{\Mspace}{\mathcal{E}_n(\mathcal{X})}
\newcommand{\Mspacei}{\mathcal{E}_{n-1}(\mathcal{X})}
\newcommand{\Prob}[1]{\mathbb{P}\left(#1\right)} 
\newcommand{\ID}[1]{ \mathbb{1} (#1 ) }      
\newcommand{\argmin}{\operatornamewithlimits{argmin}} 
\newcommand{\GDSS}{g}
\newcommand{\GNS}{g}
\usepackage{amsthm}

\newtheorem{Assumption}{Assumption}
\newtheorem{Problem}{Problem}
\newtheorem{Lemma}{Lemma}
\newtheorem{remark}{Remark}
\newtheorem{Definition}{Definition}
\newtheorem{Theorem}{Theorem}
\newtheorem{Corollary}{Corollary}
\newcommand{\edit}[1]{\textcolor{black}{#1}}

\title{\LARGE \bf   Deep Nash and Sequential Mean-Field Equilibria in Cooperative and Non-cooperative  Games  with Imperfect Information Structures 
 }

\author{Jalal Arabneydi and Amir G. Aghdam
\thanks{ This work has been supported in part by the Natural Sciences and Engineering Research Council of Canada (NSERC) under Grant RGPIN-262127-17, and in part by Concordia University under Horizon Postdoctoral Fellowship.}  
\thanks{Jalal Arabneydi and Amir G. Aghdam are with the  Department of Electrical and Computer Engineering, 
        Concordia University, 1455 de Maisonneuve Blvd, Montreal, QC, Canada. 
        {\tt\small Email:jalal.arabneydi@mail.mcgill.ca} and
        {\tt\small Email:aghdam@ece.concordia.ca}}%

Initial draft: Mar. \nth{6}, 2018.
}

\begin{document}
\maketitle
\thispagestyle{empty}
\pagestyle{empty}

\begin{abstract}
A class of  nonzero-sum stochastic dynamic games with  imperfect  information structure  is  investigated. The game involves  an arbitrary number of players, modeled as  homogeneous Markov decision processes,  aiming to find a sequential Nash equilibrium. The players  are coupled in both dynamics and  cost functions through the empirical distribution of  states and actions of players.  Two non-classical information structures are considered: deep state sharing and no-sharing, where deep state refers to the empirical distribution of the states of players.   In the former, each player observes its local state  as well as the deep state while in the latter  each player observes only  its local state.  For both finite- and infinite-horizon cost functions,  a  sequential  equilibrium, called deep Nash equilibrium,  is identified,  where the dynamics of deep state resembles a convolutional neural network.    In addition,  an approximate sequential equilibrium, called sequential mean-field equilibrium,   under  no-sharing information structure  is  proposed,  whose performance  converges to that of the deep Nash  equilibrium  despite the fact that the  strategy  is not necessarily continuous with respect to  the deep state.   The proposed strategies are robust to  trembling-hand imperfection in both microscopic and macroscopic levels. Finally, the extension to multiple sub-populations  and   arbitrarily-coupled (asymmetric) cost functions are demonstrated.
\end{abstract}
\section{Introduction}

%
%


 Nonzero-sum stochastic dynamic games are ubiquitous  in decision-making applications  such as   finance, management and smart grid, wherein a group of players compete with each other in order to minimize (maximize) their  cost (utility) functions.  It is well known that  when   every player perfectly observes the states of all players, a backward induction can be devised  to identify a subgame-perfect Nash equilibrium (called Markov perfect Nash equilibirum),   which is    a set of strategies satisfying  sequential rationality requirement~\cite{osborne1994course}.  
  In practice, however, it  is not always  possible  to have perfect information about the  states of players  for various reasons such as the cost  of information (specially when the number of players is large) and the privacy of the players. In such a case, the information structure is imperfect, which generally  results in a  phenomenon, known as the \emph{infinite regress of the compound expectation}, wherein  every player must know what others know about what he/she knows about what they know and so on~\cite{harsanyi1967games}.  In the seminal work of Harsanyi~\cite{harsanyi1967games}, a method is proposed  to stop the infinite regression by imposing the common knowledge hypothesis  among players.   The solution concept in such games, known as Bayesian games, is a pair of strategy-belief (rather than only the strategy), where  the belief and strategy are consistent.   Since  the  belief space increases exponentially with time horizon as well as  the number of players, it is computationally difficult to find a tractable solution for  such  games with more than a few players~\cite{reif1984complexity}.   In addition, it is not always feasible to find a sequential equilibrium  in dynamic games  with imperfect information due to  the fact  that the belief of every player about the states of  other players  depends on the strategy of other players, in general.   As a result,  each  player's belief does not necessarily evolve in  a Markovian manner, i.e.,  the standard  backward induction is not applicable~\cite{ouyang2017dynamic}.

Due to the above difficulties, mean-field games~\cite{HuangPeter2006,Lasry2007mean} were introduced more than a decade ago to provide an   approximate  solution  by exploiting  two key features:  negligible effect of individual players and the law of large numbers. More precisely,  since  the effect of  a single player on the infinite population is negligible,   the sequential rationality requirement  reduces  to a conventional dynamic programming decomposition, and   the  belief of every   player about other players reduces  to a common   belief (known as the mean-field)   that evolves \emph{deterministically} in time.  This type of  strategy-belief pair  is called \emph{mean-field equilibrium}, which  is presented in the form of two  coupled forward-backward equations, where the backward equation is  a Hamilton-Jacobi-Bellman equation and  the forward one is a  Fokker-Plank-Kolmogorov equation.  The consistency requirement is also established by  imposing various Lipschitz-type fixed-point conditions (that generally hold for small time horizons) or monotonicity-type assumptions (that are often difficult to verify). In order to show  that the mean-field equilibrium  constitutes  an approximate  equilibrium for the finite-population game,    the standard approach is  to assume  implicitly or explicitly that  the solution   is  continuous in the mean-field. For more details, the reader is referred to~\cite{Caines2018book} and references therein.
 

%


While the  above  results  are  interesting and  useful in some applications,   an often overlooked  question is  that   to what extent   the mean-field equilibrium is practical.   For example, the mean-field equilibrium is not a  ``sequential''  equilibrium~\cite{kreps1982sequential} or ``trembling-hand''  equilibrium~\cite{selten1975reexamination} in the sense that  it does not take into account   the off-equilibrium-path  events  occurred at the macroscopic  level.  In other words,  players   agree upon  the trajectory of the mean-field  before the game starts,  but  there is no guarantee that they  hold on  to the initially agreed-upon belief  at  every stage of the game  if an unexpected  event  changes their belief about the mean-field   (e.g.,   small common mistakes).   In addition,  the decision of  each player at any stage of the game  depends not only on the past decisions of  the players but also on the future ones, which makes  the mean-field equilibrium  future-dependent.  Another practical concern is the unnatural assumption that the strategy is continuous in the mean field.  It is argued in~\cite{doncel2016meanb} that such assumption may remove many  meaningful  equilibria. In particular,  the authors in~\cite{doncel2016meanb}  provide a counterexample in which   the tit-for-tat principle is not applicable  because the deviation of a single player from the agreed-upon equilibrium  is invisible  to   the infinite population due to the negligible effect, meaning that  the deviant player will   not be penalized by other players according to the mean-field equilibrium (i.e.,  a single player can  take advantage of other players without facing any consequences).

In this paper,  inspired by~\cite{Jalal2019MFT} and~\cite{Jalal2019Automatica}, we take a different route from the above literature and study  a  game   consisting of an arbitrary number of homogeneous players with finite state and action spaces under two imperfect  information structures: deep-state sharing and no-sharing, where  the deep state refers to the empirical distribution of  the states of a finite population  (rather than that of an infinite one). Since the number of players  is not necessarily large,  the simplification afforded by the negligible  effect  is not applicable here.  We  are  interested in  index-invariant sequential Nash equilibria  and argue that a sustainable  equilibrium in homogeneous  games  must be  index-invariant  (fair) because any index-dependant  equilibrium can cause  discrimination against some  players that is based on   the way  the players are  indexed (labeled). This   discriminatory treatment naturally leads  to protest and   anarchy.   It is to be noted that  the mean-field equilibrium  is also  index-invariant because  every player uses the  strategy of a generic player.    Given that players are interested to  reach  a fair Nash equilibrium (agreement), we first  analyze  the dynamics of  the belief of  players under  index-invariant strategies, and  then  identify  a sequential equilibrium by developing  a dynamic programming decomposition under deep-state sharing information structure. Next, we  propose  an approximate sequential equilibrium under no-sharing information structure  that converges to the deep-state sharing (finite-population) solution  as the number of players goes to infinity.  In contrast to mean-field games that compute an  infinite-population equilibrium  and then impose some kind of continuity condition on the solution to make it  applicable  to the finite-population model, we study a  finite-population model  and  propose an approximate solution  without imposing any continuity assumption on the solution.   The  equilibria proposed in this paper  are not in the   form of coupled forward-backward equations. This feature allows one to  incorporate feedback information, pertaining to  the microscopic as well as macroscopic behaviours,  into the equilibria. Furthermore,  unlike the stationary mean-field equilibrium  that assumes the mean-field is stationary and does not change with time~\cite{adlakha2015equilibria},   we do not  restrict ourselves to stationary  mean-field  because the  mean-field  normally varies with time according to the dynamics of the players. 

The remainder of this paper is organized as follows.  The problem  is formulated  in Section~\ref{sec:prob}, and  the dynamics of the deep state is then  presented in Section~\ref{sec:pre}.  An exact solution under  deep state sharing and an approximate  one under  no-sharing  are  proposed in Section~\ref{sec:DSS}.  The main results are  then  extended to the   infinite-horizon discounted cost  function  in Section~\ref{sec:infinite}.  A numerical example is presented in Section~\ref{sec:numerical}, and   the paper is concluded in Section~\ref{sec:con}.

\section{Problem Formulation}\label{sec:prob}
\subsection{Notation}
In this paper, $\mathbb{R}$, $\mathbb{R}_{\geq 0}$ and $\mathbb{N}$  refer to  real, positive real and natural
numbers, respectively.   For any $k \in \mathbb{N}$, the finite set of integers $\{1,\ldots,k\}$
is denoted by $\mathbb{N}_k$.  Furthermore, $\Prob{\boldsymbol \cdot}$   is  the probability of a random variable; $\Exp{ \boldsymbol \cdot}$ is  the expectation of an event; $\ID{\boldsymbol \cdot }$  is  the indicator function of a set; $\Ninf{\boldsymbol \cdot }$  is  the
infinity norm of a vector, and   $\delta(\boldsymbol \cdot)$  is the Dirac measure with a unit mass concentrated at a single point, specified by the argument.   For any finite set $\mathcal{X}$,  $\mathcal{E}_n(\mathcal{X})$ denotes  the space of empirical distribution of $n \in \mathbb{N}$ samples form set $\mathcal{X}$,  $\mathcal{P}(\mathcal{X})$ denotes the space of probability measures defined on $\mathcal{X}$, $|\mathcal{X}|$  denotes the cardinality of $\mathcal{X}$, and  $\Conv(A(x), \forall x \in \mathcal{X})$  denotes the convolution   of functions  $A(x)$ over  all $x \in \mathcal{X}$.  The shorthand notation $\Binopdf( n, p)$ denotes the binomial probability density function with $n \in \mathbb{N}$ trails and success probability $p \in [0,1]$. Also, the short-hand  notation $x_{1:t}$ is used to denote the set $\{x_1,\ldots,x_t\}$. Given a set of  $n \in \mathbb{N}$  components, the superscript $-i$ is used to represent   all  components except for the  $i$-th one, $i \in \mathbb{N}_n$.

\subsection{Model }
Consider a stochastic dynamic  game with $n \in \mathbb{N}$ homogeneous players.  Denote by $x^i_t \in \mathcal{X}$ and $u^i_t \in \mathcal{U}$, respectively,   the state and action of player $i \in \mathbb{N}_n$ at time $t \in \mathbb{N}$.  Denote by $\mathfrak{D}_t \in \mathcal{E}_n(\mathcal{X} \times \mathcal{U})$   the empirical distribution of states and actions of players  at time $t$ and  by $d_t \in \Mspace$ the empirical distribution of states of players,  where  for any $x \in \mathcal{X}, u \in \mathcal{U}$ and $t \in \mathbb{N}$:
\begin{align}\label{eq:def-MF-n}
\mathfrak{D}_t(x,u)&= \frac{1}{n} \sum_{i=1}^n \ID{x^i_t=x} \ID{u^i_t=u},\, \nonumber  \\
d_t(x)&= \frac{1}{n} \sum_{i=1}^n \ID{x^i_t=x}.
\end{align}
The dynamics of the state of  player $i \in \mathbb{N}_n$   at time $t \in \mathbb{N}$ is influenced by  other players through the aggregate behavior $d_t$ as follows:
\begin{equation}\label{eq:dynamics-general}
x^i_{t+1}=f_t(x^i_t,u^i_t,d_t,w^i_t),
\end{equation}
where $w^i_t \in \mathcal{W}$ is  the local noise of player $i$ at time $t$. Each player $i \in \mathbb{N}_n$ selects action $u^i_t \in \mathcal{U}$  according to the probability distribution $\gamma^i_t(x^i_t)$, where $\gamma^i_t:\mathcal{X} \rightarrow \mathcal{P}(\mathcal{U})$. For the special case of pure strategies, $\gamma^i_t(x^i_t)=\delta(u^i_t)$.  Alternatively, the dynamics~\eqref{eq:dynamics-general} can be expressed  in terms of  transition probability matrix as:
\begin{align}\label{eq:dynamics-prob}
&\mathcal{T}_t(x^i_{t+1},x^i_{t}, \gamma^i_t(x^i_t),d_t):=\Prob{x^i_{t+1}\mid x^i_t,\gamma^i_t(x^i_t),d_t}\\
&= \sum_{u \in \mathcal{U}} \Prob{x^i_{t+1}\mid x^i_t,u,d_t} \gamma^i_t(x^i_t)(u) \\
&=\sum_{u \in \mathcal{U}}\sum_{w \in \mathcal{W}} \ID{x^i_{t+1}=f_t(x^i_t,u,d_t, w)}\Prob{w^i_t=w} \gamma^i_t(x^i_t)(u).
\end{align} 
In the sequel, the two equivalent representations~\eqref{eq:dynamics-general} and~\eqref{eq:dynamics-prob} are occasionally interchanged for  ease of display.  Let $\mathbf x_t:=\{x_t^1,\ldots,x^n_t\}$, $\mathbf u_t:=\{u_t^1,\ldots,u^n_t\}$ and  $\mathbf w_t:=\{w_t^1,\ldots,w^n_t\}$,  $t \in \mathbb{N}$.  It is assumed that the primitive random variables $\{\mathbf x_1, \mathbf w_1,\ldots,\mathbf w_T\}$ are defined on a common probability space and are mutually independent. In addition,  $\mathcal{X}$,  $\mathcal{U}$ and $ \mathcal{W}$ are   finite sets  in  the Euclidean space. Furthermore, the initial state $x^1_1,\ldots,x^n_1$ are i.i.d. random variables with probability mass function $P_X$, and the local  noises $w^1_t,\ldots,w^n_t$ are i.i.d. random variables with probability mass function $P_{W_t}$, for any $ t \in \mathbb{N}_T$.

\subsection{Admissible strategies}
In the sequel,  we  refer to  the empirical distribution of states as \emph{deep state}, which is inspired by  the fact that  its transition probability matrix  resembles a convolutional neural network~\cite{Jalal2019MFT}. We consider two non-classical  information structures:    deep state sharing  information structure (DSS)  and  no-sharing information structure (NS).  In the DSS, every player  has access to  its local state and  the history of the deep state at  any time $t$, i.e., action $u^i_t \in \mathcal{U}$ is selected according to the following probability distribution:
\begin{equation}\label{eq:info-DSS}
u^i_t \sim \GDSS^{i}_t(x^i_t,d_{1:t}), \quad i \in \mathbb{N}_n, t \in \mathbb{N}, 
\end{equation}
where $\GDSS^i_t: \mathcal{X} \times (\Mspace)^t \rightarrow \mathcal{P}(\mathcal{U})$ is the control law at time~$t$. In practice, there are various applications
in which DSS is plausible. For example,  in the stock markets the   players (i.e. buyers, sellers and brokers) are often provided with statistical data on the total share value with some statistics on trades and exchanges. Also,  in a smart grid, an independent service operator may collect and broadcast the aggregate demand  in the grid. It is also possible, under certain conditions,  to obtain the deep state without a central authority using a suitable consensus algorithm. For instance, under some connectivity conditions,  robots in a swarm can obtain the deep state in a distributed manner by interacting with
their neighbours.   In the NS, on the other hand,  every  player  has access only to   its local state at any time instant, i.e., 
\begin{equation}\label{eq:info-ns}
u^i_t \sim \GNS^i_t(x^i_t), \quad i \in \mathbb{N}_n, t \in \mathbb{N},
\end{equation}
where $\GNS^i_t:\mathcal{X} \rightarrow \mathcal{P}(\mathcal{U})$.  When the number of players is very large, NS information structure  is more  practical as it requires no communication between players after the initial time.

\begin{Definition}[\textbf{Index-invariant (fair) strategy}]\label{def1}
For any $i \in \mathbb{N}_n$ and $t \in \mathbb{N}_T$,  let $I^i_t$ denote the information available to player $i $ by  time $t$. The strategies of players $i$ and $j$ are said to be index-invariant if  
 $ g^i_t(I^i_t)=g^j_t(\sigma_{i,j} I^j_t)$, $i,j \in \mathbb{N}_n$, $t \in \mathbb{N}_T$, where  $g^i_t$ is a generic control law of player $i$ at time $t$ and the operator $\sigma_{i,j}$  swaps information pertaining to players $i$ and~$j$. 
\end{Definition}
 
Let $\mathbf \GDSS^i=\{\GDSS^i_1,\ldots,\GDSS^i_T\}$ denote the  strategy of player $i \in \mathbb{N}_n$   over the control horizon $T \in \mathbb{N}$.   For any $i \in \mathbb{N}_n$ and $t \in \mathbb{N}_T$,  let  also $c_t: \mathcal{X} \times \mathcal{U} \times \mathcal{E}_n(\mathcal{X} \times \mathcal{U}) \rightarrow \mathbb{R}_{\geq 0}$ denote the per-step cost of player $i $ at time $t $ consisting of non-cooperative and cooperative  costs:
\begin{equation}
c_t(x^i_t,u^i_t,\mathfrak{D}_t)= c_{\text{non-cooperative}}+c_{\text{cooperative}},
\end{equation}
where  
\begin{align}
c_{\text{non-cooperative}}&:=\breve c_t(x^i_t,u^i_t,d_t), \\
c_{\text{cooperative}}&:=\frac{1}{n}\sum_{i=1}^n \bar c_t(x^i_t,u^i_t,d_t), 
\end{align}
such that $\breve c_t, \bar c_t: \mathcal{X} \times \mathcal{U} \times \Mspace \rightarrow \mathbb{R}_{\geq 0}$.
 Define the following  total expected cost for player $i \in \mathbb{N}_n$:
\begin{equation}\label{eq:cost_player}
{J^i_n(\mathbf g^i,\mathbf g^{-i})}_{t_0}= \mathbb{E} [\sum_{t=t_0}^T c_t(x^i_t,u^i_t,\mathfrak{D}_t)],\quad  t_0 \in \mathbb{N}_T,
\end{equation}
where the above expectation  is taken with respect to the probability measures induced by  the choice of  players' strategies.

\begin{Definition}[\textbf{Deep Nash Equilibrium (DNE)}]\label{def:DSS}
Any strategy $\{\mathbf \GDSS^{\ast},\ldots,\mathbf \GDSS^{\ast} \}$ under DSS information structure is said to be a deep Nash equilibrium if for any player $i \in \mathbb{N}_n$ at any stage of the game $t_0\in \mathbb{N}_T$:
\begin{equation}
{J^i_n(\mathbf \GDSS^{\ast},\mathbf \GDSS^{\ast})}_{t_0} \leq {J^i_n(\mathbf \GDSS^{ i},\mathbf \GDSS^{\ast})}_{t_0},
\end{equation}
where $\mathbf g^i$ is any arbitrary  DSS strategy.
\end{Definition}
  It is worth highlighting that the solution concept of deep teams~\cite{Jalal2019MFT} is different from Nash equilibrium, in general. 
  
\begin{Definition}[\textbf{Sequential Mean-Field Equilibrium (SMFE)}]\label{def:pDSS}
Any  strategy $\{\mathbf {\GNS},\ldots,\mathbf {\GNS} \}$ under NS information structure  is said to be a sequential mean-field equilibrium  if  for any player  $i \in \mathbb{N}$ at any stage of the game $t_0 \in \mathbb{N}_T$:
\begin{itemize}
\item It is an infinite-population equilibrium:
\begin{equation}
{J^i_\infty(\mathbf \GDSS,\mathbf \GDSS)}_{t_0} \leq {J^i_\infty(\hat{\mathbf g}^i,\mathbf \GDSS)}_{t_0}, \quad \forall \hat{\mathbf g}^i.
\end{equation}
\item Its performance  converges to that of  a deep Nash equilibrium asymptotically:
\begin{equation}
| {J^i_n(\mathbf g,\mathbf g)}_{t_0} - {J^i_n(\mathbf {\GNS}^{\ast},\mathbf {\GNS}^{\ast})}_{t_0}| \leq  \varepsilon(n)
\end{equation}
 where $\varepsilon(n)$ is a sequence converging to zero as $n \rightarrow \infty$. 
 \end{itemize}
\end{Definition}

\begin{remark}
\emph{Note that deep Nash equilibrium in Definition~\ref{def:DSS} is a sequential equilibrium  for any arbitrary number of players (not necessarily large). On the other hand, sequential mean-field equilibrium  in Definition~\ref{def:pDSS}   is only meaningful for the case   when the number of players is  very  large. In general,  sequential mean-field equilibrium is different from the conventional  forward-backward mean-field equilibrium (FBMFE), which is a non-sequential equilibrium.  In particular, }
\end{remark}

\begin{enumerate}
\item SMFE is a sequential equilibrium  in the sense that it takes into account both on- and off-equilibrium-path events. In contrary,   FBMFE  is not a sequential one (at the mean-field level); however, it offers a  rather different setting  wherein, for example,  the cost can be non-Markovian  with respect to the  mean-field~\cite{kizilkale2019integral}.  Since the computational complexity of SMFE and FBMFE are different, in general,  they  often  have diverse applications. For instance,  SMFE is more desirable for long-horizon games and  reinforcement learning problems (as the complexity of the solution in the forward-backward solution increases exponentially with the horizon)   while FBMFE   is more suitable for games with large state spaces (as the complexity of computing the SMFE increases  exponentially with respect to the cardinality of the state space).

\item SMFE  is not necessarily an $\varepsilon(n)$-Nash equilibrium, i.e., 
\begin{equation}
{J^i_n(\mathbf \GDSS,\mathbf \GDSS)}_{t_0=1} \nleq {J^i_n(\hat{\mathbf g}^i,\mathbf \GDSS)}_{t_0=1} +  \varepsilon(n), \quad \forall \hat{\mathbf g}^i.
\end{equation}
This is because of the fact that an infinite-population solution is not necessarily the limit of the finite-population solution; see a counterexample in~\cite{doncel2016meanb} that shows the infinite-population  game may admit many meaningful equilibria (that do  not coincide with FBMFE).  In general, for the SMFE to be an $\varepsilon(n)$-Markov-Nash equilibrium,    not only the model but also  the solution must be continuous  with respect to  mean-field.  On the other hand,   it is often difficult to verify conditions imposed on the solution because the solution is unknown a priori; however,  for some special cases such as linear quadratic models~\cite{Jalal2019Automatica},  the continuity of the solution  is shown to be  without loss of optimality.


\end{enumerate}

\begin{Problem}\label{prob:DSS}
Find a deep Nash equilibrium (DNE) under DSS information structure.
\end{Problem}

\begin{Problem}\label{prob:pDSS}
Find a sequential mean-field equilibrium (SMFE) under NS information structure.
\end{Problem}

\subsection{Main  contributions}

The main contributions of this paper are outlined below.
\begin{enumerate}

\item We  present some fundamental properties of   finite-population games with an  arbitrary number of homogeneous players (not necessarily large), wherein players wish to find an index-invariant (fair) equilibrium. In particular, the structure of the transition probability matrix of the deep-state of the other players  is obtained in~Theorem~\ref{thm:dynamics_iid} in terms of the convolution function  of some binomial probability distributions, which proves to be useful for computational purposes. 

\item We study the cooperative and non-cooperative games   in a unified framework for both  finite and infinite horizon cases   with  finite and infinite  number of players. In particular, we identify a sequential equilibrium under DSS  information structure in Theorem~\ref{thm:fair} and an approximate sequential equilibrium under NS information structure in Theorem~\ref{thm:finite-ns-convergence}. The proposed results   can be extended  to asymmetric cost functions according to  Remark~\ref{remark:arbitrarily-coupled}.

\item We develop  two Bellman equations for the infinite-horizon discounted  cost functions  in Theorems~\ref{thm:inf-DSS} and~\ref{thm:inf-pDSS}. It is to be noted that the approximate  equilibrium proposed in Theorem~\ref{thm:inf-pDSS} is not stationary because  the belief system varies with time, implying that  the standard   stationary assumptions in~\cite{tembine2009mean,adlakha2015equilibria} are rather restrictive. 

\item We establish the rate of convergence for both finite- and infinite-horizon cost functions as the number of players  goes to infinity,  without  imposing any continuity assumption on the solution,  where  the optimality gap is  defined as the distance  to the deep Nash equilibrium   rather than the  infinite-population equilibrium.  In addition, a quantization scheme is proposed whose cost of computation converges to zero as the quantization level increases (Remark~\ref{remark-difference}). Note that the  proposed approximate equilibrium  does not necessarily converge to  the conventional mean-field equilibrium  since  it  is not necessarily continuous with respect to  the mean-field. 


\item Since  our proposed solutions are not future-dependent,  their extension to multiple sub-populations  is rather straightforward compared to  multiple sub-population mean-field games. For example,   it is demonstrated in~\cite{Nourian2013MM} that the extension of mean-field game approach to major-minor case is conceptually difficult because  the  trajectory of  the mean-field of minor players becomes  stochastic  due to the randomness of a non-negligible (major) player, implying that the future is unpredictable. Such complexity, however, does not arise  in our solutions as  they are  independent of  the future decisions. For more details on the extension to multiple sub-populations, the reader is referred to a similar argument presented in   deep teams~\cite{Jalal2019MFT}.
\end{enumerate}

\section{Dynamics of finite-population game}\label{sec:pre}

\begin{Lemma}\label{lemma:fair}
Suppose  that players $i$ and $j$, $i,j  \in \mathbb{N}_n$, use an index-invariant strategy.  Then,  $g^i_t=g^j_t$, $ t \in \mathbb{N}_T$, under   DSS  and NS information structures.
\end{Lemma}
\begin{proof}
The proof follows immediately  from   Definition~\ref{def1} and equations~\eqref{eq:info-DSS} and \eqref{eq:info-ns}. 
\end{proof}

Denote by $\mathcal{G}$  the space of all functions with   domain $\mathcal{X}$ and range $\mathcal{P}(\mathcal{U})$.  Suppose that  all players except  player $i \in \mathbb{N}_n$   use an index-invariant DSS strategy. Then, one   can decompose  the  strategy  into  local and global control laws  as follows:
\begin{equation}\label{eq:def-gamma-DSS}
\psi^{-i}_t(d_{1:t}):= \GDSS_t^{-i}(\boldsymbol \cdot, d_{1:t}, g^{-i}_{1:t-1},g^i_{1:t-1}), \hspace{.2cm}
\gamma^{-i}_t:= \psi^{-i}_t(d_{1:t}),
\end{equation}
for any $i \in \mathbb{N}_n$ and $t \in \mathbb{N}_T$.  Similarly, for  NS information structure one has:
\begin{equation}\label{eq:def-gamma-nfs}
\psi^{-i}_t:= \GNS_t^{-i}(\boldsymbol \cdot, g^{-i}_{1:t-1},g^i_{1:t-1}),\quad 
\gamma^{-i}_t:= \psi^{-i}_t(\boldsymbol \cdot ).
\end{equation}
The dependence of global laws $\psi^{-i}_t$  on  the past strategies of players  is explicitly displayed in equations~\eqref{eq:def-gamma-DSS} and~\eqref{eq:def-gamma-nfs}. It is to be noted that although the local control laws $\gamma^{-i}_t:  \mathcal{X}\rightarrow \mathcal{P}(\mathcal{U})$ under DSS and NS are different, the set  $\mathcal{G}$ is identical for both information structures, i.e.,
\begin{equation}\label{eq:action_random_p}
 \Prob{u^j_t \mid x^j_t}=\gamma^{-i}_t(x^j_t),\quad \forall j \in \mathbb{N}_n, j \neq i. 
\end{equation}

 Now, on noting that  $\mathfrak{D}^{-i}_t$ and $d^{-i}_t$  are the empirical distributions of  all players except  player $i $ at time $t $, i.e., for any $x \in \mathcal{X}$ and $u \in \mathcal{U}$:
\begin{align}\label{eq:def-MF-n-1}
\mathfrak{D}^{-i}_t(x,u)&=\frac{1}{n-1}\sum_{j \neq i \in \mathbb{N}_n} \ID{x^j_t=x} \ID{u^j_t=u},\nonumber\\
d^{-i}_t(x)&=\frac{1}{n-1}\sum_{j \neq i \in \mathbb{N}_n} \ID{x^j_t=x},
\end{align}
 it  follows  from~\eqref{eq:def-MF-n} and~\eqref{eq:def-MF-n-1} that for any $i \in \mathbb{N}_n$ and $t \in \mathbb{N}_T$:
\begin{align}
\frac{n-1}{n}\mathfrak{D}^{-i}_t+\frac{1}{n} \delta(x^i_t,u^i_t)&= \mathfrak{D}_t, \nonumber \\
\frac{n-1}{n}d^{-i}_t+\frac{1}{n} \delta(x^i_t)&= d_t.
\end{align}
Given any state $x^i_t  \in \mathcal{X}$ of player $i \in \mathbb{N}_n$ at time $t \in \mathbb{N}$, define functions $Q_{t|x^i_t}$ and $F_{t \mid x^i_t}$ for any $y,x \in \mathcal{X}$, $\tilde d \in \Mspacei$ and $\gamma \in \mathcal{G}$  as follows:
\begin{equation}
\begin{cases}
Q_{t|{x^i_t}}(y,x,\tilde d,\gamma):= \ID{ \tilde d(x)=0} \delta(0) +\ID{\tilde d(x) > 0} \\
\quad \times \Binopdf( (n-1)\tilde d(x), \mathcal{T}_t(y,x,\gamma(x), \frac{n-1}{n}\tilde d +\frac{1}{n}\delta( x^i_t) ) ),\\
F_{t\mid x^i_t}(y,\tilde d,\gamma):=\Conv(Q_{t\mid x^i_t}(y,x,\tilde d,\gamma), \forall x \in \mathcal{X}),
\end{cases}
\end{equation}
where $Q_{t\mid x^i_t}: \mathcal{X}^2 \times \Mspacei \times \mathcal{G}  \rightarrow \mathcal{P}(0,1, \ldots, (n-1)\tilde d(x))$ and  $F_{t\mid x^i_t}: \mathcal{X} \times \Mspacei \times \mathcal{G}  \rightarrow \mathcal{P}(0,1,\ldots, n-1)$.
\begin{Theorem}\label{thm:dynamics_iid}
 If all players except player $i \in \mathbb{N}_n$  use  the local law $\gamma^{-i}_t \in \mathcal{G}$ under DSS or NS information structure at time $t \in \mathbb{N}_T$, then the transition probability matrix of their deep state  for any $y \in \mathcal{X}$ and $k \in \mathbb{N}_{n}$ is given by:
\begin{equation}
\mathbb{P}( d^{-i}_{t+1}(y)=\frac{k-1}{n-1}\mid x^i_t, d^{-i}_{t},  \gamma^{-i}_t)=F_{t|x^i_t}(y,d^{-i}_t,\gamma^{-i}_t)(k).
\end{equation}
In addition,
\begin{multline}\label{eq:expectation_theorem}
\Exp{d^{-i}_{t+1} (y) \mid x^i_t, d^{-i}_t, \gamma^{-i}_t}= \sum_{x \in \mathcal{X}} d^{-i}_t(x)\\
\times \mathcal{T}_t(y, x, \gamma^{-i}_t(x),  \frac{n-1}{n} d^{-i}_t + \frac{1}{n} \delta(x^i_t)).
\end{multline}
\end{Theorem}
\begin{proof}
It follows from~\eqref{eq:def-MF-n-1} that for every $y \in \mathcal{X}$,
\begin{align}\label{eq:mean-field_i_proof}
& \hspace{-.5cm}(n-1)d^{-i}_{t+1}(y)=\sum_{j=1,j \neq i}^n \ID{x^j_{t+1}=y} \nonumber \\
&\hspace{-.5cm}=\sum_{x \in \mathcal{X}} \sum_{j=1,j\neq i}^n \ID{x^j_t=x} \ID{f_t(x,\gamma^{-i}_t(x),d_t,w^j_t)=y}.
\end{align}
For any $x \in \mathcal{X}$, the inner summation  in right-hand side of~\eqref{eq:mean-field_i_proof}  consists of $(n-1)$ components, where $(n-1)(1-d^{-i}_t(x))$ components are zero  due to the fact that there are only $(n-1)d^{-i}_t(x)$ components that have state $x$ at time $t$, according to the definition of the empirical distribution $d^{-i}_t(x)$. These $(n-1)$ possibly non-zero components are independent binary random variables  with the following success probability:
\begin{multline}
\mathbb{P}(\ID{f_t(x,\gamma^{-i}_t(x),d_t,w^j_t)=y}=1)=\mathcal{T}_t(y,x,\gamma^{-i}_t(x),d_t).
\end{multline}
Let $F_{t\mid x^i_t}(y,d^{-i}_t,\gamma^{-i}_t)$ denote the probability mass function (PMF) of the sum of these  $(n-1)d^{-i}_t(x)$ Bernoulli random variables,
which is a binomial distribution with $(n-1)d^{-i}_t(x)$ trials and success
probability $\mathcal{T}_t(y,x,\gamma^{-i}_t(x),\frac{n-1}{n}d^{-i}_t+\frac{1}{n}\delta(x^i_t))$. Now, the PMF of $(n-1)d^{-i}_{t+1}(y)$ is the PMF of the outer summation in the right-hand side of equation~\eqref{eq:mean-field_i_proof} that  consists of $|\mathcal{X}|$ independent random variables, each
of which has the PMF $Q_{t\mid x^i_t}(y,x,d^{-i}_t,\gamma^{-i}_t)$.  Therefore, the PMF of $(n-1)d^{-i}_{t+1}(y)$  can be expressed as the convolution of the PMFs $Q_{t\mid x^i_t}(y,x,d^{-i}_t,\gamma^{-i}_t)$ over space $\mathcal{X}$.

In addition, it follows from~\eqref{eq:def-MF-n},~\eqref{eq:dynamics-prob} and~\eqref{eq:def-MF-n-1}  that  for every $y \in \mathcal{X}$ and $t \in \mathbb{N}_T$:
\begin{align}
&\Exp{d^{-i}_{t+1}(y) \mid x^i_t, d^{-i}_t, \gamma^{-i}_t}= \Exp{ \frac{1}{n}\sum_{i=1}^n \ID{x^i_{t+1}=y}\mid x^i_t, d^{-i}_t, \gamma^{-i}_t}
\\
&=\frac{1}{n} \sum_{i=1}^n  \Exp{\ID{x^i_{t+1}=y}\mid x^i_t, d^{-i}_t, \gamma^{-i}_t}=\sum_{x \in \mathcal{X}} d^{-i}_t(x)\\
&\quad \times  \mathcal{T}_t(y, x, \gamma^{-i}_t(x), \frac{n-1}{n} d^{-i}_t + \frac{1}{n} \delta(x^i_t)),
\end{align}
where  the states of all players except that of player $i$ has  identical transition probability.
\end{proof}

\begin{Lemma}\label{lemma:markov-mf}
 If all players except player $i \in \mathbb{N}_n$  use a fair strategy  under  the DSS and NS, the following holds  irrespective of the strategy~$\mathbf g^i$ at any time $t \in \mathbb{N}_T$,
\begin{align}\label{eq:Markov-property}
&\Prob{x^i_{t+1}, d_{t+1} \mid x^i_{1:t},d_{1:t}, \gamma^i_{1:t}, \gamma^{-i}_{1:t}} \nonumber \\
&\hspace{.2cm}=\Prob{x^i_{t+1} \mid x^i_{t}, \gamma^i_{t}(x^i_t), d_t} \Prob{d^{-i}_{t+1} \mid x^i_{t},d^{-i}_{t},  \gamma^{-i}_{t}}.
\end{align}
\end{Lemma}
\begin{proof}
The proof follows from  Lemma~\ref{eq:def-MF-n-1} and equations~\eqref{eq:dynamics-general},~\eqref{eq:dynamics-prob} and \eqref{eq:def-MF-n-1}. 
\end{proof}

\begin{Lemma}\label{lemma:joint_empirical}
Suppose  all players except player $i \in \mathbb{N}_n$ use a local  law $\gamma^{-i}_t \in \mathcal{G}$ at time $t \in \mathbb{N}_T$. Then, the following holds for any $i,k \in \mathbb{N}_n$,  $t \in \mathbb{N}_T$,  $x \in \mathcal{X}$ and $u \in \mathcal{U}$:
\begin{multline}
\Prob{\mathfrak{D}^{-i}_t(x,u)=\frac{k-1}{n-1} \mid x^i_{1:t},d^{-i}_{1:t},\gamma^i_{1:t},\gamma^{-i}_{1:t}}\\=\Binopdf((n-1)d^{-i}_t(x), \gamma^{-i}_t(x)(u))(k),
\end{multline}
and
\begin{equation}\label{eq:joint-lemma}
\Exp{\mathfrak{D}^{-i}_t(x,u)\mid x^i_{1:t},d^{-i}_{1:t},\gamma^i_{1:t},\gamma^{-i}_{1:t}}=d^{-i}_t(x)  \gamma^{-i}_t(x)(u).
\end{equation}
\end{Lemma}
\begin{proof}
The proof  is presented in Appendix~\ref{sec:proof_lemma:joint_empirical}.
\end{proof}

\begin{Lemma}\label{lemma:per_step_cost}
Suppose  all players except player $i \in \mathbb{N}_n$ use a local  law $\gamma^{-i}_t \in \mathcal{G}$ at time $t \in \mathbb{N}_T$. There exists a function $\ell_t: \mathcal{X} \times \Mspace \times \mathcal{P}(\mathcal{U}) \times \mathcal{G}: \rightarrow \mathbb{R}_{\geq 0}$ such that the per-step cost function of player $i$ can be expressed as follows:
\begin{align}
&\Exp{c_t(x^i_t,u^i_t,\mathfrak{D}_t)\mid x^i_{1:t}, d_{1:t},\gamma^{i}_{1:t},\gamma^{-i}_{1:t}}\\
&=\sum_{u,\mathfrak{D}^{-i}} c_t(x^i_t,u, \frac{n-1}{n}\mathfrak{D}^{-i}+ \frac{1}{n}\delta(x^i_t,u))\gamma^{i}_t(x^i_t)(u) \\
&\quad \times \Prob{\mathfrak{D}^{-i}_t=\mathfrak{D}^{-i} \mid d^{-i}_{t},\gamma^{-i}_{t}}
=:\ell_t(x^i_t, d_t, \gamma^i_t(x^i_t),\gamma^{-i}_t).
\end{align}
\end{Lemma}
\begin{proof}
The proof follows from Lemma~\ref{lemma:joint_empirical}, equation~\eqref{eq:def-MF-n-1} and the fact that $u^i_t$ is distributed according to the probability distribution $\gamma^i_t(x^i_t)$.
\end{proof}

In the next lemma, it is shown  that the belief  of every player can be characterized by the deep state of other players.
\begin{Lemma}\label{lemma:conditional_prob}
 Let  all players except player $i \in \mathbb{N}_n$ use the local  law $\gamma^{-i}_t \in \mathcal{G}$ at time $t \in \mathbb{N}_T$.  Then,   under DSS or NS information structure,  irrespective of the  strategies of the  players, the following relation holds:
\begin{equation}\label{eq:conditional_prob_non}
\Prob{\mathbf x^{-i}_t \mid  x^i_{1:t},  d^{-i}_{1:t},\gamma^{i}_{1:t}, \gamma^{-i}_{1:t}}=\Prob{\mathbf x^{-i}_t \mid   d^{-i}_{t}}.
\end{equation}
\end{Lemma}
\begin{proof}
The proof is presented in Appendix~\ref{sec:proof_lemma:conditional_prob}.
\end{proof}
 The following lemma is a consequence of the above result.
\begin{Lemma}\label{lemma:cost}
If  all players except player $i \in \mathbb{N}_n$ use  the local law $\gamma^{-i}_t \in \mathcal{G}$ at time $t$, then there exists a function $\ell_t: \mathcal{X}  \times \Mspace\times  \mathcal{P}(\mathcal{U}) \times \mathcal{G}\rightarrow \mathbb{R}_{\geq 0}$ such that for any arbitrarily-coupled cost function $c_t(x^i_t,u^i_t, \mathbf{x}_t, \mathbf u_t): \mathcal{X} \times \mathcal{U} \times \mathcal{X}^n \times \mathcal{U}^n \rightarrow \mathbb{R}_{\geq 0}$,
\begin{equation}
\Exp{c_t(x^i_t,u^i_t, \mathbf{x}_t, \mathbf u_t) \mid  x^i_{1:t}, d_{1:t}, \gamma^i_{1:t},  \gamma^{-i}_{1:t}}=:\ell_t(x^i_t, d_t,\gamma^i_t(x^i_t),\gamma^{-i}_t).
\end{equation}
\end{Lemma}
\begin{proof}
The proof is presented in Appendix~\ref{sec:proof_lemma:cost}.
\end{proof}

\begin{remark}
\emph{Note that $d^{-i}_t$ is not necessarily a Markov process, i.e.
$\Prob{d^{-i}_{t+1}|d^{-i}_{1:t}, \gamma^{-i}_{1:t}}\neq \Prob{d^{-i}_{t+1}|d^{-i}_{t}, \gamma^{-i}_{t}}$.
When $n=\infty$, however,  the above inequality becomes equality due to the  negligible effect. }
\end{remark}

\begin{remark}
\emph{Note that the results of Theorem~\ref{thm:dynamics_iid} and Lemmas~\ref{lemma:markov-mf}--\ref{lemma:cost}  hold irrespective of control laws $g^i_{1:t}$ and $\psi^{-i}_{1:t}$, $t \in \mathbb{N}_T$.}
\end{remark}

\section{Finite Horizon}\label{sec:DSS}
\subsection{Solution of Problem~\ref{prob:DSS}}
 Define value  functions $V_{T+1}, V_T,\ldots,V_1$ such that for any $(x^i_{T+1},d_{T+1}) \in \mathcal{X} \times \Mspace$,  $V_{T+1}(x^i_{T+1}, d_{T+1})=0$, and for any $t \in \mathbb{N}_T$ and $(x^i_{t},d_{t} ) \in \mathcal{X}\times \Mspace$,
\begin{multline}\label{eq:dp_value-i}
V_t(x^i_t,d_t)=\min_{\gamma^i_t(x^i_t) \in \mathcal{P}(\mathcal{U})}( \ell_t(x^i_t, d_t, \gamma^{i}_t(x^i_t),\gamma^{-i}_t)\\
+ \mathbb{E}[ V_{t+1}(x^i_{t+1},d_{t+1}) \mid x^i_t, d_t, \gamma^i_t(x^i_t), \gamma^{-i}_t] ), 
\end{multline}
where  $\ell_t$ is given  by Lemma~\ref{lemma:per_step_cost} and $\gamma^{-i}_t$ is the local law of all players except  player~$i$ at time $t$.
\begin{Lemma}\label{thm:2}
 If   all players  except player $i \in \mathbb{N}_n$ use the same DSS strategy, then   the best-response strategy  for player~$i$ is obtained from dynamic program~\eqref{eq:dp_value-i}.
\end{Lemma}
\begin{proof}
The proof follows from the fact that  the stochastic  process $(x^i_{t+1}, d_{t+1})$ evolves in a Markovian manner under $\gamma^i_t(x^i_t)$ and $\gamma^{-i}_t$, according to Lemma~\ref{lemma:markov-mf}. In addition,  from Lemma~\ref{lemma:per_step_cost}, it follows that the expected per-step cost  of player~$i$ at time~$t$ can be described by $\ell_t(x^i_t,d_t,\gamma^i_t(x^i_t),\gamma^{-i}_t)$. Since the results of Lemmas~\ref{lemma:markov-mf} and~\ref{lemma:per_step_cost} do not depend on  $g^i_{1:t}$ and $\psi^{-i}_{1:t}$,  $(x^i_t,d_t)$ is an information state for player $i$. Subsequently, one can write  the dynamic program~\eqref{eq:dp_value-i} in order to find the best-response strategy of player~$i$.
\end{proof}
\begin{remark}\label{remark:arbitrarily-coupled}
According to Lemma~\ref{lemma:cost},  dynamic programming decomposition  proposed in  Lemma~\ref{thm:2} extends to any arbitrarily-coupled (asymmetric) cost function.
\end{remark}

In the next theorem, it is shown that deep Nash equilibrium  always exists.
\begin{Theorem}\label{thm:fair}
Problem~\ref{prob:DSS} always admits an index-invariant solution  that satisfies  the best-response equation~\eqref{eq:dp_value-i} across all players, simultaneously.
%
\end{Theorem}

\begin{proof}
For ease of display,  we present the  best-response equation~\eqref{eq:dp_value-i} as follows:
\begin{equation}\label{eq:Best}
\gamma^i_t= B_t(d_t)(\gamma^{-i}_t),
\end{equation}
where for any $t \in \mathbb{N}_T$ and $d_t \in \Mspace$, $B_t(d_t): \mathcal{G} \rightarrow \mathcal{G}$, i.e.,
\begin{align}\label{eq:Best_1}
B_t(d_t)&:=\argmin_{\gamma^i_t(\boldsymbol \cdot)} ( \ell_t(\boldsymbol \cdot, d_t, \gamma^{i}_t(\boldsymbol \cdot),\gamma^{-i}_t)\\
&\quad + \mathbb{E}[ V_{t+1}(x^i_{t+1},d_{t+1}) \mid \boldsymbol \cdot, d_t, \gamma^i_t(\boldsymbol \cdot), \gamma^{-i}_t] ).
\end{align}
 We now  show  that  equation~\eqref{eq:Best} admits at least one  index-invariant fixed-point solution at each time $t \in \mathbb{N}_T$, i.e., $\gamma_t:=\gamma^i_t=\gamma^{-i}_t$.  At any time $t \in \mathbb{N}_T$, given any $x^i_t$, $d_t$ and $\gamma^{-i}_t$, the  solution of  equation~\eqref{eq:dp_value-i} can be expressed as follows:
 \begin{align}
 &\argmin_{\gamma^i_t(x^i_t)} \ell_t(x^i_t,d_t, \gamma^i_t(x^i_t),\gamma^{-i}_t) \\
 &+ \sum_{x^i_{t+1}, d_{t+1}} \Prob{x^i_{t+1},d_{t+1}\mid x^i_t, d_t, \gamma^i_t(x^i_t), \gamma^{-i}_t} V_{t+1}(x^i_{t+1},d_{t+1})\\
& \substack{(a)\\=} \argmin_{\gamma^i_t(x^i_t)} \ell_t(x^i_t,d_t, \gamma^i_t(x^i_t), \gamma^{-i}_t)\\
 &+\sum_{x^i_{t+1},d^{-i}_{t+1}} \big[\sum_{u} \Prob{x^i_{t+1} \mid x^i_t, u, d_t}\gamma^i_t(x^i_t)(u) \big] \\
& \times \Prob{d^{-i}_{t+1}\mid d^{-i}_t, \gamma^{-i}_t} V_{t+1}(x^i_{t+1}, \frac{n-1}{n} d^{-i}_{t+1} + \frac{1}{n} \delta(x^i_{t+1})),
 \end{align}
 where $(a)$ follows from Lemma~\ref{lemma:markov-mf} and equations~\eqref{eq:dynamics-prob} and~\eqref{eq:def-MF-n-1}.   Note that the argument of the above equation is piece-wise linear in $\gamma^{i}_t(x^i_t)$ due to Lemma~\ref{lemma:per_step_cost}. Hence, the minimization problem is a convex optimization, i.e.,  $B_t(d_t)(\gamma^{-i}_t)$ is a non-empty and  convex set. In addition,   $B_t(d_t)(\boldsymbol \cdot)$ is a closed graph because it is continuous with respect to $\gamma^{-i}_t$ according to Theorem~\ref{thm:dynamics_iid} and Lemmas~\ref{lemma:joint_empirical} and~\ref{lemma:per_step_cost}, on noting that binomial probability distribution is continuous with respect to the success probability. Since $\mathcal{G}$ is a  non-empty, compact and convex subset of a locally convex Hausdorff space,  the set-valued mapping $B$ has a fixed-point solution~\cite[Chapter 17]{charalambos2006infinite}. This means that there exists a local law $\gamma_t \in \mathcal{G}$ such that:
$
 \gamma_t=B_t(d_t)(\gamma_t).
 $
\end{proof}
In the  view of Theorem~\ref{thm:fair}, we remove the subscript $i$ in the sequel  because  the deep Nash equilibrium, identified by the dynamic program~\eqref{eq:dp_value-i},  is index-invariant.  Note that the equilibrium still depends on the number of players $n$. Consequently,  we define a \emph{generic} player with the same dynamics and per-step cost as any individual  player, which competes against $n-1$ identical players of its own kind.  Let $x_t \in \mathcal{X}$ and $\gamma_t \in \mathcal{G}$ denote  the state and local law of the generic player  at time $t \in \mathbb{N}_T$, respectively, and  define  the value  functions $V_{T+1}, V_T,\ldots,V_1$  such that for any $(x_{T+1},d_{T+1} ) \in \mathcal{X} \times  \Mspace$,  $V_{T+1}(x_{T+1}, d_{T+1})=0$, and for any $t \in \mathbb{N}_T$ and  $(x_{t},d_{t}) \in \mathcal{X}  \times  \Mspace$,
\begin{multline}\label{eq:DP_DSS}
V_t(x_{t}, d_{t})=\min_{\gamma_t(x_t) \in  \mathcal{P}(\mathcal{U})}(  \ell_t(x_t, d_t,\gamma_t(x_t),\gamma_t)\\
+ \mathbb{E}[V_{t+1}(x_{t+1}, d_{t+1}) \mid x_{t}, d_t, \gamma_t(x_t), \gamma_t]).
\end{multline}

\subsection*{Infinite population: A special case}
Consider a special case in which the number of players is infinite, i.e. $n =\infty$. In this case,  deep state $d_t \in \Mspace$ reduces to mean field $m_t  \in \mathcal{P}(\mathcal{X})$, $ t \in \mathbb{N}_T$, where  the average of any  infinite number of  i.i.d.  binary random variables is  equal to their expectation, almost surely, according to the strong law of large numbers. Consequently, the dynamics of the deep state  can be simplified  for $n=\infty$, according to \eqref{eq:expectation_theorem}. In particular,  define a vector-valued function $\hat f_t: \mathcal{P}(\mathcal{X}) \times \mathcal{G} \rightarrow \mathcal{P}(\mathcal{X})$, $t \in \mathbb{N}_T,$ such that for any $m \in \mathcal{P}(\mathcal{X})$ and $\gamma \in \mathcal{G}$, 
\begin{equation}\label{eq:def-hat-f}
\hat f_t(m,\gamma):= \sum_{x \in \mathcal{X}} m(x) \mathcal{T}_t(\boldsymbol \cdot, x, \gamma(x), m),
\end{equation}
where $m_1:=P_X$, and for any $t \in \mathbb{N}_T$:
$m_{t+1}:=\hat f_t(m_t, \gamma_t)$.   In addition,  from Lemma~\ref{lemma:per_step_cost} and equations~\eqref{eq:def-MF-n-1} and~\eqref{eq:joint-lemma},  define  the infinite-population per-step cost function $\hat \ell_t:  \mathcal{X} \times \mathcal{P}(\mathcal{X}) \times \mathcal{P}(\mathcal{U}) \times \mathcal{G}$ at time $t \in \mathbb{N}_T$ as follows:
\begin{align}\label{eq:cost_infinite}
\begin{cases}
\hat \ell_t(x_t, m_t, \gamma_t(x), \gamma_t):=\sum_{u \in \mathcal{U}} c_t(x_t,u, \mathfrak{M}_t) \gamma_t(x)(u),  \\
\mathfrak{M}_t(x,u):=m_t(x) \gamma_t(x)(u), x \in \mathcal{X}, u \in \mathcal{U},
\end{cases}
\end{align}
where $\mathfrak{M}_t \in \mathcal{P}(\mathcal{X} \times \mathcal{U})$ denotes the infinite-population  $\mathfrak{D}_t \in \mathcal{E}_\infty(\mathcal{X} \times \mathcal{U})$. Finally, define value functions $\hat V_{T+1}, \hat V_T,\ldots,\hat V_1$ such that for any $(x_{T+1},m_{T+1} ) \in \mathcal{X} \times  \mathcal{P}(\mathcal{X})$,  $\hat V_{T+1}(x_{T+1}, m_{T+1})=0$, and for any  $t \in \mathbb{N}_T$ and    $(x_t,m_t) \in \mathcal{X} \times  \mathcal{P}(\mathcal{X})$,
\begin{multline}\label{eq:hat-V}
\hat V_t(x_t,m_t)=\min_{\gamma_t(x_t) \in \mathcal{P}(\mathcal{U})}(  \hat \ell_t(x_t, m_t,\gamma_t(x), \gamma_t)\\
 +\mathbb{E}[\hat V_{t+1}(x_{t+1},m_{t+1}) \mid x_t, m_t,  \gamma_t(x),\gamma_t]).
\end{multline}

\begin{Corollary}\label{cor:infinite}
 Given  any mean-field $m_t \in \mathcal{P}(\mathcal{X})$ at any time $t \in \mathbb{N}_T$, the fixed-point equation~\eqref{eq:hat-V} always has a  solution.
\end{Corollary}
\begin{proof}
The proof follows along the  same steps of the proof of Theorem~\ref{thm:fair}, where  deep state $d_t$ simplifies  to mean field $m_t$ with the dynamics~\eqref{eq:def-hat-f} and cost~\eqref{eq:cost_infinite}.
\end{proof}

\begin{remark}
\emph{
In mean-field games~\cite{HuangPeter2006,gomes2010discrete,Lasry2007mean}, mean-field refers to the infinite-population limit of the deep state (i.e. $\Exp{d_t}=m_t$)  and  in  mean-field-type game~\cite{carmona2013control,bensoussan2013mean,carmona2018probabilistic,Saldi2018},  it refers to the probability distribution of the state of the generic player (i.e., $\Exp{d_t}=\Prob{x_t}$). Note that  the solution concept of the mean-field-type game is generally different from   Nash equilibrium,  because  $m_t \neq \Prob{x_t}$.   For the  cooperative cost function with decoupled dynamics,~\cite{Sanjari2019}  proposes  some convexity conditions  under which the infinite-population cooperative (Nash bargaining) solution coincides with the team-optimal solution.  For the special case of linear quadratic games, the reader is referred to~\cite{Jalal2019Automatica} for  similarities and differences between mean-field games~\cite{huang2007large}, mean-field-type games~\cite{Bensoussan2016} and  deep teams~\cite{Jalal2019risk}.}
\end{remark}

%
%
%
%
\subsection{Solution of Problem~\ref{prob:pDSS}}

To  propose an approximate solution for Problem~\ref{prob:pDSS}, we make the  following mild  assumption on the model.
\begin{Assumption}\label{assumption: finite-lipschitz}
There exist  constants $K^p_t, K^c_t \in \mathbb{R}_{ \geq 0}$, $t \in \mathbb{N}_T$, (that do not depend on $n$) such that for every $x,y \in \mathcal{X}$, $u \in \mathcal{U}$, $d,m \in \mathcal{P}(\mathcal{X})$ and $\mathfrak{D}, \mathfrak{M} \in \mathcal{P}(\mathcal{X} \times \mathcal{U})$,
\begin{align}
| \Prob{y|x,u,d} - \Prob{y|x,u,m} | &\leq K^p_t\|d- m\|,\\
| c_t(x,u,\mathfrak{D}) - c_t(x,u,\mathfrak{M}) | &\leq K^c_t\|\mathfrak{D}- \mathfrak{M}\|.
\end{align}
\end{Assumption}

\begin{remark}
\emph{Note that Assumption~\ref{assumption: finite-lipschitz} is not much restrictive, and holds for any function that is polynomial  in $d$ and $\mathfrak{D}$ due to the fact that  they are confined to   the bounded domains $\mathcal{P}(\mathcal{X})$ and $\mathcal{P}(\mathcal{X} \times \mathcal{U})$, respectively. Furthermore, any continuous function can be approximated  by polynomial functions  as closely as desirable according  to  Weierstrass Theorem.}
\end{remark}

Denote by $s_t(x_t,m_t)\in  \mathcal{P}(\mathcal{U})$, $t \in \mathbb{N}_T,$  a solution of the fixed-point equation~\eqref{eq:hat-V}.  Define the following NS strategy:
\begin{equation}\label{eq:ns_strategy}
\gamma_t= s_t(\boldsymbol \cdot ,m_t), \quad t \in \mathbb{N}_T,
\end{equation}
where $m_1=P_X$, and  for any $t \in \mathbb{N}_T$,
\begin{equation}\label{eq:update_z_s}
m_{t+1}=\hat f_t(m_t, s_t(\boldsymbol \cdot, m_t)).
\end{equation}

\begin{remark}\label{remark:compute-z}
\emph{ It is to be noted that given  the  strategy $s_{1:T}$  and  probability mass function of initial states $P_X$,   $m_{1:T}$ can be  calculated  by    every player independently, according to~\eqref{eq:update_z_s}. In addition, if the players  happen to commonly change their belief about the mean-field at any stage of the game, the dynamics of the  belief system~\eqref{eq:update_z_s}  can accommodate this trembling-hand effect as the control law $s_t$ is in the state-feedback form. }
\end{remark}

%

\begin{Lemma}\label{lemma:evolution-lip}
Let Assumption~\ref{assumption: finite-lipschitz} hold. Let also  $d_t \in \Mspace$, $\mathfrak{D}_t \in \mathcal{E}(\mathcal{X}\times \mathcal{U})$,  $m_t \in \mathcal{P}(\mathcal{X})$ and $\mathfrak{M}_t \in \mathcal{P}(\mathcal{X} \times \mathcal{U})$  be  controlled by  the same local control law $\gamma_t \in \mathcal{G}$. Then, there exists a constant $K^m_t  \in \mathbb{R}_{ \geq 0}$ such that at any time $t \in \mathbb{N}_T$:
\begin{align}
\Exp{\Ninf{d_{t+1} - m_{t+1}} } \leq K^m_t \Ninf{d_t -m_t} + \mathcal{O}(\frac{1}{\sqrt{n}}),\\
\Exp{\Ninf{\mathfrak{D}_{t+1} - \mathfrak{M}_{t+1}} } \leq K^m_t \Ninf{d_t -m_t} + \mathcal{O}(\frac{1}{\sqrt{n}}),
\end{align}
where $ \mathcal{O}(\frac{1}{\sqrt{n}})$ does not depend on the control horizon $T$.
\end{Lemma}
\begin{proof}
The proof is presented in Appendix~\ref{sec:proof_lemma:evolution-lip}. 
\end{proof}

\begin{Lemma}\label{lemma:uniform_bounded_cost}
Let Assumption~\ref{assumption: finite-lipschitz} hold. For any $x \in \mathcal{X}$, $\gamma(x) \in \mathcal{G}$,  $m \in \mathcal{P}(\mathcal{X})$ and $t \in \mathbb{N}_T$, one has: $\ell_t(x,m,\gamma(x), \gamma) \leq K^{c}_t$. 
\end{Lemma}
\begin{proof}
The proof follows from Lemma~\ref{lemma:per_step_cost}, Assumption~\ref{assumption: finite-lipschitz} and the fact that spaces $\mathcal{X}$ and $\mathcal{U}$ are finite and $\|\mathfrak{D}- \mathfrak{M}\| \leq 1, \forall \mathfrak{D},\mathfrak{M} \in \mathcal{P}(\mathcal{X} \times \mathcal{U})$.
\end{proof}
%
%
%

Define the following  non-negative constants backward in time such that for any $t \in \mathbb{N}_{T}$:
\begin{align}\label{eq:kv-ko}
K^v_t&:=K^c_t+K^v_{t+1}K^m_t+ K^p_t  \sum_{\tau=1}^{t+1} \beta^{\tau-1}K^c_{\tau},\nonumber  \\
 K^o_t&:=K^v_{t+1}+K^o_{t+1},
\end{align}
where $K^v_{T+1}=K^o_{T+1}=0$.  


\begin{Lemma}\label{lemma:Lipschitz}
Let Assumption~\ref{assumption: finite-lipschitz} hold. For  any  $x \in \mathcal{X}$, $d_t \in \Mspace$, $m_t \in \mathcal{P}(\mathcal{X})$ and  $t \in \mathbb{N}_T$, the following inequality holds:
\begin{equation}\label{eq:inequality-value}
|V_t(x,d_t) - \hat V_t(x, m_t)| \leq K^v_t \Ninf{d_t -m_t} + K^o_t \mathcal{O}(\frac{1}{\sqrt n}),
\end{equation}
where $ \mathcal{O}(\frac{1}{\sqrt{n}})$ does not depend on the control horizon $T$.
\end{Lemma}
 \begin{proof}
The proof is presented in Appendix~\ref{sec:proof_lemma:Lipschitz}.
 \end{proof}
 
 Let $\hat x_t$ and $\hat d_t$ denote the state and deep-state of the  generic player at time $t \in \mathbb{N}_T$ under the proposed  NS strategy \eqref{eq:ns_strategy}.
 \begin{Lemma}
 Let Assumption~\ref{assumption: finite-lipschitz} hold.  Given any $\hat d_t \in \Mspace$ and  $m_t \in \mathcal{P}(\mathcal{X})$,  $t \in \mathbb{N}_T$, the following inequality holds:
\begin{align}
|\Exp{\sum_{\tau=1}^t \sum_{u}c_t(\hat x_t,u, \hat{\mathfrak{D}}_t) s_t(\hat x_t,m_t)(u)} 
-\Exp{\sum_{\tau=1}^t \sum_{u}c_t(x_t,u, \mathfrak{M}_t)\\
\times  s_t( x_t,m_t)(u) } | \leq  K^v_t \Ninf{\hat d_t -m_t} + K^o_{t} \mathcal{O}(\frac{1}{\sqrt n}).
\end{align}
 \end{Lemma}
 \begin{proof}
 The proof  is similar to the proof of Lemma~\ref{lemma:Lipschitz}, on noting  that $\mathfrak{D}_t$ and $\mathfrak{M}_t$ are governed by the same control law $s_t(\boldsymbol \cdot, m_t)$, $t \in \mathbb{N}_T$.
  \end{proof}
 
\begin{Theorem}\label{thm:finite-ns-convergence}
Let  Assumption~\ref{assumption: finite-lipschitz}   hold,  and  also  equations~\eqref{eq:DP_DSS} and~\eqref{eq:hat-V} admit a unique solution. Then,  $\{s_t(\boldsymbol \cdot, m_t),m_t\}_{t=1}^T$ is a solution of Problem~\ref{prob:pDSS}, i.e.
$|J^\ast_n - \hat J_n | \in \mathcal{O}(\frac{1}{\sqrt n})$,
where  $J^\ast_n$ and $\hat J_n$  are the performance values   of the  generic player under the solutions satisfying~\eqref{eq:DP_DSS} and~\eqref{eq:hat-V}, respectively.
\end{Theorem}

\begin{proof}
The proof is presented in Appendix~\ref{sec:proof_thm:finite-ns-convergence}.
\end{proof}

%
\begin{remark}\label{remark-difference}
\emph{\edit{ It is to be noted that the dynamic program~\eqref{eq:hat-V}  may involve  a non-smooth non-convex optimization over   $\mathcal{P}(\mathcal{X})$,  which  is an uncountably  infinite set.  To overcome this hurdle,   one may replace  $\mathcal{P}(\mathcal{X})$  by a  finite (quantized)  space  similar to that proposed in~\cite[Corollary 1]{JalalCDC2017} and~\cite[Theorems 4 and 5]{Jalal2019MFT}.  An immediate implication is that if the quantization level is  $\sqrt n$, the resultant quantization solution converges to the sequential mean-field equilibrium at the same that the unquantized solution does (i.e. $1/ \sqrt n$)}. }
\end{remark}

\section{Infinite Horizon} \label{sec:infinite}
In this section, we extend our main results to the  infinite horizon discounted cost.  To this end, it is assumed that the model described in Section~\ref{sec:prob} is time-homogeneous; hence, the subscript $t$ is omitted  from the notation.  Denote by $\beta \in (0,1)$   the discount factor and  by $J^{i,\beta}_n$  the total expected discounted cost for player $i \in \mathbb{N}_n$, i.e.
\begin{equation}
J^{i,\beta}_n:=\Exp{\sum_{t=1}^\infty \beta^{t-1} c(x^i_t,u^i_t,\mathfrak{D}_t)}.
\end{equation}
 Define also the infinite-horizon counterpart of Bellman equation~\eqref{eq:DP_DSS} such that for  any $(x,d) \in \mathcal{X}\times  \Mspace$,  
 
 \begin{multline}\label{eq:Bellman_DSS}
V(x, d)=\min_{\gamma (x) \in  \mathcal{P}(\mathcal{U})}(  \ell(x, d,\gamma(x),\gamma)\\
+\beta \sum_{x^+ \in \mathcal{X}, d^+ \in \Mspace} \Prob{x^+, d^+| x,d,\gamma} V(x^+,d^+).
\end{multline}


\begin{Theorem}\label{thm:inf-DSS}
  The Bellman equation~\eqref{eq:Bellman_DSS} admits a solution, and that solution  is  a sequential equilibrium  for   the infinite-horizon discounted cost function under DSS information structure.
\end{Theorem}
\begin{proof}  
From Lemma~\ref{thm:2}, the best response strategy of any  player $i \in \mathbb{N}_n$  is given by~\eqref{eq:dp_value-i} for any finite horizon $T \in \mathbb{N}$.  Define  a real-valued function $W^i_t$ for any $ i \in \mathbb{N}_n$ and $t \in \mathbb{N}_T$,  $(x,d) \in \mathcal{X} \times  \Mspace$ as follows:
\begin{equation}\label{eq:W-def}
W^i_t(x,d):=\beta^{-T+t-1} V^i_{T-t+2}(x,d), 
\end{equation}
where $W^i_1(x,d):= \beta^{-T} V^i_{T+1}(x,d)=0$. It can be shown that:
\begin{multline}\label{eq:proof_infinite_DSS}
W^i_{T+1}(x^i,d)=\min_{\gamma^i(x^i) \in \mathcal{P}(\mathcal{U})} ( \ell(x^i,d,\gamma^i(x^i),\gamma^{-i})  \\
+\beta \sum_{ x^{i,+} \in \mathcal{X},d^+ \in \Mspace} \Prob{x^{i,+}, d^+|x,d, \gamma(x^i),\gamma^{-i}} W_{T}( x^{i,+}, d^+)),
\end{multline}
where $W^i_{T+1}(x^i,d)=V_1^i(x^i,d)$. For any $T \in \mathbb{N}$ and $d \in \Mspace$, define the best-response function $B_T(d): \mathcal{G} \rightarrow \mathcal{G}$, i.e.,
\begin{align}\label{eq:Best_2}
&B_T(d):=\argmin_{\gamma^i(\boldsymbol \cdot)} ( \ell(\boldsymbol \cdot, d, \gamma^{i}(\boldsymbol \cdot),\gamma^{-i})\\
&+\beta \sum_{ x^{i,+} \in \mathcal{X},d^+ \in \Mspace} \Prob{x^{i,+}, d^+|\boldsymbol \cdot ,d, \gamma^i(\boldsymbol \cdot),\gamma^{-i}} W_{T}( x^{i,+}, d^+)).
\end{align}
 Given any $x^i$, $d$ and $\gamma^{-i}$, the  solution of  equation~\eqref{eq:proof_infinite_DSS} can be expressed as follows:
 \begin{align}
&\argmin_{\gamma^i(x^i)} \ell(x^i,d, \gamma^i(x^i), \gamma^{-i})\\
 &+\beta \sum_{x^{i,+},d^{-i,+}} \big[\sum_{u} \Prob{x^{i,+} \mid x^{i}, u, d}\gamma^i(x^i)(u) \big] \\
& \times  \Prob{d^{-i,+}\mid d^{-i}, \gamma^{-i}} W_{T}(x^{i,+}, \frac{n-1}{n} d^{-i,+} + \frac{1}{n} \delta(x^{i,+})).
 \end{align}
From  Lemma~\ref{lemma:per_step_cost}, it results that the argument of the above equation is piece-wise linear in $\gamma^{i}(x^i)$;  hence, the  above minimization  is a convex optimization, i.e.,  $B_T(d)(\gamma^{-i})$ is a non-empty and  convex set. In addition,   $B_T(d)(\boldsymbol \cdot)$ is a closed graph because it is continuous with respect to $\gamma^{-i}$ according to Theorem~\ref{thm:dynamics_iid} and Lemmas~\ref{lemma:joint_empirical} and~\ref{lemma:per_step_cost}, on noting that binomial probability distribution is continuous with respect to its success probability. Since $\mathcal{G}$ is a  non-empty, compact and convex subset of a locally convex Hausdorff space,  the set-valued mapping $B$ has a fixed-point solution~\cite[Chapter 17]{charalambos2006infinite}. This means that there exists a local law $\gamma \in \mathcal{G}$ such that:
$
 \gamma=B_T(d)(\gamma).
 $

On the other hand, since  the discount factor $\beta$ is less than one,  the Bellman equation~\eqref{eq:proof_infinite_DSS} is a contractive  mapping with respect to  the infinity norm, implying that~\eqref{eq:proof_infinite_DSS} converges to a  solution, i.e., for  every $(x,d) \in \mathcal{X} \times \Mspace$,
\begin{equation}\label{eq:w-inf}
\lim_{T \rightarrow \infty } W_{T+1}(x,d)=W_\infty(x,d)=:V(x,d).
\end{equation}
\end{proof}

For any $x \in \mathcal{X}$ and $m \in \mathcal{P}(\mathcal{X})$, define the infinite-horizon counterpart of the dynamic program~\eqref{eq:hat-V}  as:
\begin{align}\label{eq:hat-V-inf}
&\hat V(x,m)=\min_{\gamma(x) \in \mathcal{P}(\mathcal{U})} ( \ell(x,m, \gamma(x),\gamma) \nonumber  \\
& \quad +\beta  \sum_{x^+ \in \mathcal{X}} \Prob{x^+| x, \gamma(x),m} V(x^+, \hat f(m,\gamma))).
\end{align}
Denote by $s(x,m)$ any solution of the fixed point equation~\eqref{eq:hat-V-inf}, and  define the following NS strategy:
\begin{equation}\label{eq:ns_strategy-inf}
\gamma^i_t= s(x^i_t,m_t), \quad i \in \mathbb{N}_n, t \in \mathbb{N},
\end{equation}
where $m_1=P_X$, and  for any $t \in \mathbb{N}$,
$
m_{t+1}=\hat f(m_t, s(\boldsymbol \cdot, m_t)).
$
\begin{remark}
\emph{Note  that  strategy~\eqref{eq:ns_strategy-inf} is not stationary  with respect to the local state as the process $\{m_t\}_{t=1}^\infty$ has dynamics.}
\end{remark}

\begin{Assumption}\label{assump:infinite_horizon}
Let $\beta K^m <1$,  where $K^m$ is given by Lemma~\ref{lemma:evolution-lip}.
\end{Assumption}

 \begin{Lemma}\label{lemma:lipschtiz-inf-V-hat-V}
 Let Assumptions~\ref{assumption: finite-lipschitz} and \ref{assump:infinite_horizon} hold.  Given any $x \in \mathcal{X}$, $d_1 \in \Mspace$ and $m_1 \in \mathcal{P}(\mathcal{X})$,   the relative distance $ | V(x, d_1) - \hat V(x, m_1)|$ is upper bounded by:
 \begin{align}
 | V(x, d_1) - \hat V(x, m_1)| \leq & \frac{(1-\beta + K^p) K^c}{1-\beta K^m} \Ninf{d_1-m_1}\\
 &\quad + \frac{1-\beta + K^p}{(1-\beta)} \frac{ K^c}{1-\beta K^m} \mathcal{O}(\frac{1}{\sqrt n}).
 \end{align}
  \end{Lemma}
 \begin{proof}
Given  the dynamic program~\eqref{eq:hat-V} and any finite horizon $T \in \mathbb{N}$,   define the non-negative real-valued function $\hat W_t (x,m)$ for  any $x \in \mathcal{X}$, $m \in \mathcal{P}(\mathcal{X})$ and $t \in \mathbb{N}_T$  as follows:
 \begin{equation}\label{eq:hat-W-def}
\hat W_t(x,m):=\beta^{-T+t-1} \hat V_{T-t+2}(x,m),
\end{equation}
where $\hat W_1(x,m):= \beta^{-T} \hat V_{T+1}(x,m)=0$.   Then,  one can obtain the following equality by simple algebraic manipulations:
\begin{multline}\label{eq:proof_infinite_pDSS}
\hat W_{T+1}(x,m)=\min_{\gamma(x) \in \mathcal{P}(\mathcal{X})} ( \ell(x,m,\gamma(x),\gamma) + \\
\beta \sum_{ x^+ \in \mathcal{X}} \Prob{\tilde x| x, \gamma(x),m}  \hat W_{T}( x^+, \hat f(m,\gamma) )),
\end{multline}
where $\hat W_{T+1}(x,m)=\hat V_1(x,m)$. Since~\eqref{eq:proof_infinite_pDSS} is  contractive, it admits a unique solution for any $(x,m) \in \mathcal{X} \times \mathcal{P}(\mathcal{X})$, i.e.
\begin{equation}\label{eq:hat-w-inf}
\lim_{T \rightarrow \infty } \hat W_{T+1}(x,m)=\hat W_\infty(x,m)=: \hat V(x,m).
\end{equation}
Now, define the following constants based on the ones given in~\eqref{eq:kv-ko}:
\begin{equation}\label{eq:kv-ko-hat}
\hat K^v_{t}:= \beta^{-T+t-1} K^v_{T-t+2}, \quad \hat K^o_{t}:= \beta^{-T+t-1} K^o_{T-t+2}, 
\end{equation}
where $\hat K^v_1:=\beta^{-T} K^v_{T+1}=0$ and $\hat K^o_1:=\beta^{-T} K^o_{T+1}=0$. It is straightforward to show that: 
\begin{equation}\label{eq:k_v-hat}
\begin{cases}
K^v_1=\hat K^v_{T+1}= K^c+ \beta \hat K^v_T K^z+K^p \sum_{\tau=1}^2 \beta^{\tau -1} K^c\\
\qquad \leq  K^c(1+\frac{1}{1-\beta}K^p) + \beta K^m \hat K^v_T,\\
K^o_1=\hat K^o_{T+1}= \beta ( \hat K^v_{T}+ \hat K^o_{T})=\sum_{\tau=1}^T \beta^{T-\tau+1} \hat K^v_{\tau} 
\\
 \qquad \leq \hat K^v_T \sum_{\tau=1}^T \beta^{T-\tau +1}.
\end{cases}
\end{equation}
From   Lemma~\ref{lemma:Lipschitz} and  equations~\eqref{eq:W-def},~\eqref{eq:hat-W-def} and~\eqref{eq:kv-ko-hat}, for  any $T \in \mathbb{N}$, the following inequality holds:
\begin{multline}\label{eq:proof-W-V}
| W_{T+1}(x,d_1) - \hat W_{T+1}(x,m_1) |=| V_{1}(x,d_1) - \hat V_{1}(x,m_1) |
\\
\leq \hat K^v_{T+1} \Ninf{d_1 - m_1} + \hat K^o_{T+1} \mathcal{O}(\frac{1}{\sqrt n}).
\end{multline}
From Assumption~\ref{assump:infinite_horizon} and equations~\eqref{eq:w-inf},~\eqref{eq:hat-w-inf},~\eqref{eq:k_v-hat} and~\eqref{eq:proof-W-V},  when  $T  \rightarrow \infty$   the following inequality is obtained: $| V(x,d_1) - \hat V(x,m_1) | \leq \hat K^v_{\infty} \Ninf{d_1 - m_1} + \hat K^o_{\infty} \mathcal{O}(\frac{1}{\sqrt n})$.
 \end{proof}
  Let  $\hat x_t$ and $\hat d_t$  denote, respectively,   the state  and deep-state of the generic player at time $t \in \mathbb{N}_T$ under the proposed NS strategy~\eqref{eq:ns_strategy-inf}, where $\hat x_1=x_1$ and $\hat d_1=d_1$. Then, the performance of  the player is given by
\begin{equation}\label{eq:proof_hat_J-inf}
  \hat J^\beta_n:=\Exp{\sum_{t=1}^\infty \beta^{t-1} \sum_{u \in \mathcal{U}}c(\hat x_t, u,\hat d_t) s(\hat x_t,m_t)(u)}.  
\end{equation}

 \begin{Theorem}\label{thm:inf-pDSS}
Let Assumptions~\ref{assumption: finite-lipschitz}--\ref{assump:infinite_horizon} hold,  and the Bellman equations~\eqref{eq:Bellman_DSS} and~\eqref{eq:hat-V-inf} admit a unique solution.  Then,   $\{s(\boldsymbol \cdot,m_t),m_t\}_{t=1}^\infty$  is a solution of Problem~\ref{prob:pDSS} with the infinite-horizon discounted cost function such that
\begin{equation}
|J^{\ast,\beta}_n - \hat J^{\beta}_n | \leq   \frac{(2-\beta )(1-\beta + K^p) }{1-\beta} \frac{ K^c}{1-\beta K^m}\mathcal{O}(\frac{1}{\sqrt n}),
\end{equation}
where  $J^{\ast,\beta}_n$  and $\hat J^{\beta}_n$  are the performance values   of the generic player under the solutions satisfying~\eqref{eq:Bellman_DSS}  and~\eqref{eq:hat-V-inf}, respectively.
 \end{Theorem}

\begin{proof}
From the triangle inequality, it results that
\begin{equation}\label{eq:triangle-inf}
|J^{\ast,\beta}_n - \hat J^\beta_n | \leq |J^{\ast,\beta}_n -\Exp{\hat V(x_1,z_1)}|+ |\hat J^\beta_n -\Exp{\hat V(x_1,z_1)}|.
\end{equation}
The first term of the right-hand side of~\eqref{eq:triangle-inf} is bounded by $\mathcal{O}(\frac{1}{\sqrt n})$ according to Lemma~\ref{lemma:lipschtiz-inf-V-hat-V}, the monotonicity of the expectation operator,   the relation $J^{\ast\beta}_n=\Exp{V(x_1,d_1)}$,  and the fact that  $m_1$ converges to $m_1=P_X$ at the rate $\mathcal{O}(\frac{1}{n})$ in the mean-square sense.  The second  term of the right-hand side of~\eqref{eq:triangle-inf} is  also bounded by a similar bound $\mathcal{O}(\frac{1}{\sqrt n})$. The existence of an index-invariant strategy for~\eqref{eq:hat-V-inf} can be established following  similar steps in Theorem~\ref{thm:inf-DSS} and  Corollary~\ref{cor:infinite}. 
 The proof is now completed, on noting that $\mathcal{O}(\frac{1}{\sqrt n}) + \mathcal{O}(\frac{1}{\sqrt n})=\mathcal{O}(\frac{1}{\sqrt n})$.
\end{proof}

\begin{Corollary}\label{cor:decoupled}
The result of Theorem~\ref{thm:inf-pDSS} holds irrespective of  Assumption~\ref{assump:infinite_horizon} if the dynamics of  the players are decoupled.
\end{Corollary}
\begin{proof}
he proof follows from the fact that when the dynamics of  players are decoupled,   $K^m=1$ in Lemma~\ref{lemma:evolution-lip} and $K^p=0$ in Assumption~\ref{assumption: finite-lipschitz}. 
\end{proof}

\begin{remark}
Similar to Remark~\ref{remark-difference}, one can use a quantized space  for the infinite-horizon cost function wherein the quantization error is upper bounded by the constant  proposed in Theorem~\ref{thm:inf-pDSS}.
\end{remark}

 \section{Numerical Example}\label{sec:numerical}
 
\textbf{Example 1.} Consider $n$ players ($n \in \mathbb{N}$)  sharing a common resource, e.g.,  a communication channel.  Each player independently makes a request  with probability $p \in (0,1)$  to have access to the resource at  any  time $t \in \mathbb{N}$. Let $q \in (0,1)$ denote the probability according to which  a request is served. If a player has a pending request, it is not allowed to send another request until its current request is served. Denote by $x^i_t \in \{0,1\}$ the state of player $i \in \mathbb{N}_n$ at time $t \in \mathbb{N}$, where $x^i_t=1$ means that player $i$ has a request at time $t$ and $x^i_t=0$ means  it has no request.

%
%

The common resource is provided by a third-party company whose profit depends on the number of requests (the higher number of requests the more profit). Let $\alpha \in \mathbb{N}_n$ denote a threshold above which  the company makes a reasonable profit, and if  the number of requests is less than   $\alpha $,   each player  has to pay a fee $c_{\text{underload}} \in \mathbb{R}_{>0}$. On the other hand,  when  the number of requests   is larger than a threshold $\gamma \in \mathbb{N}_n$, players may experience  some discomfort such as delay  in accessing the resource. Denote by $c_{\text{overload}} \in \mathbb{R}_{>0}$ the cost associated with an overload of requests. At each time instant, there are three options available to players: (1) everyone sends a request  without any commitment to others; (2) everyone  commits to  decrease  the number of requests,  and (3)  everyone commits to increase the number of requests.  Denote by $u^i_t$ the action of player $i$ at time $t$, and let $u^i_t=1,2,3$ be respectively the action corresponding to the options (1)--(3) described above. The transition probability of each player $i \in \mathbb{N}_n$ under action $u^i_t=1$ is given by
\begin{equation}
\Prob{x^i_{t+1}\mid x^i_t,u^i_t=1}=
\left[\begin{array}{cc}
1-p & p\\
q & 1-q
\end{array}\right].
\end{equation}
Let $ p_{D} \leq p$ denote the  probability of request when players agree  to decrease the number of requests. In such a case,  players  with pending requests drop them with some probability. Let $q_D \geq q$ denote the probability  that a request is not pending  (either served or dropped). Therefore, the transition probability of player  $i $ under action  $u^i_t=2$ is described by
\begin{equation}
\Prob{x^i_{t+1}\mid x^i_t,u^i_t=2}=
\left[\begin{array}{cc}
1-p_D & p_D\\
q_D & 1-q_D
\end{array}\right].
\end{equation}
Denote by $ p_{I} \geq p$ the  probability of request when players agree  to  increase the number of requests. Hence, the transition probability of player  $i $ under action $u^i_t=3$ is expressed by
\begin{equation}
\Prob{x^i_{t+1}\mid x^i_t,u^i_t=3}=
\left[\begin{array}{cc}
1-p_I & p_I\\
q & 1-q
\end{array}\right].
\end{equation}
\edit{Since the state space is binary,  the empirical distribution of one state is sufficient to identify  that of the other state. Hence,  with a slight abuse of notation, denote  $d_t$ as the  empirical distribution of the requests of  all players  at time $t \in \mathbb{N}$, i.e.}
\begin{equation}
d_t=\frac{1}{n} \sum_{i \in \mathbb{N}_n} \ID{x^i_t=1}.
\end{equation}
If player $i$ wishes  to selfishly use the shared resource  without taking the states of other players  into account,  others  can  penalize that player  by    sending either a small number of requests resulting in $c_{\text{underload}}$ or a large number of requests leading to $c_{\text{overload}}$ as follows:
\begin{equation}
c(x^i_t,d^{-i}_t)=\begin{cases}
c_{\text{underload}}, & (n-1) d^{-i}_{t}  < \alpha -x^i_t,\\
c_{\text{overload}}, & (n-1) d^{-i}_t \geq \gamma -x^i_t.
\end{cases}
\end{equation}
Given a discount factor $\beta \in (0,1)$, define
\begin{equation}
J^{i,\beta}_n=\Exp{ \sum_{t=1}^\infty \beta^{t-1} c(x^i_t,d^{-i}_t)}.
\end{equation}
\edit{For the sake of transparency, the company announces   the empirical distribution of requests at each time instant, i.e.,   the information structure is deep-state sharing. The  objective of  the players is to reach a fair agreement (Nash strategy)  among themselves to efficiently utilize the shared resource.} 
   Figure~\ref{fig1}  displays a  Nash strategy for  the following numerical parameters:
\begin{align}
&n=100, p=0.3, q=0.3, p_D=0.2, q_D=0.4, p_I=0.4,\\
 &\beta=0.9, \alpha=30,\gamma=70, 
c_{\text{underload}}=5, c_{\text{overload}}=1.
\end{align}
 The decision of each player $i \in \mathbb{N}_n$ at time $t \in \mathbb{N}$ depends on  the local state $x^i_t$  and  \edit{the empirical distribution of the requests of other players} $d^{-i}_t$.  It is shown in  Figure~\ref{fig1}  that \edit{the trajectory of  the number of  requests of  players   lies between the lower and upper bounds for different initial states.}
\begin{figure}
\centering
\hspace{-.2cm}
\includegraphics[trim={0 10cm 0 10cm},clip, width=\linewidth]{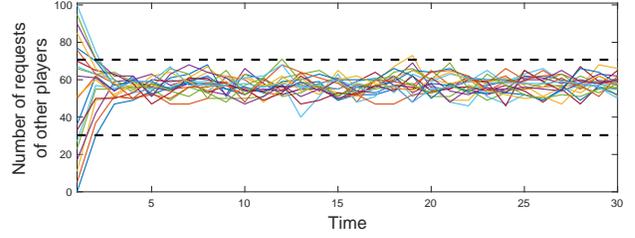}
\caption{A trajectory  of the number of requests in  Example 1, given   $20$  different initial states.}\label{fig1}
\end{figure}

\section{Conclusions} \label{sec:con}
In this paper,   a game consisting of a set of homogeneous players  wishing to reach an index-invariant  (fair) Nash equilibrium  was studied. The players were modeled as controlled Markov chains, where  their  dynamics   and cost functions  were coupled through the empirical distribution of their states (deep state). Two non-classical information structures, namely deep-state sharing and no-sharing information structures, were investigated. Since the number of players was finite  (and not necessarily large), the effect of a single player on other players was non-negligible and the deep state  was a random vector (rather than a deterministic one).    A sequential equilibrium was identified under  the deep-state sharing information structure and an approximate  one   was proposed under the no-sharing structure  for both finite- and infinite-horizon cost functions.

\bibliographystyle{IEEEtran}
\bibliography{Jalal_Ref}

\appendices

\section{Proof of  Lemma~\ref{lemma:joint_empirical}}\label{sec:proof_lemma:joint_empirical}
From equation~\eqref{eq:def-MF-n-1},  one has:
\begin{equation}
(n-1) \mathfrak{D}^{-i}_t(x,u)=\sum_{j\neq i}^n \ID{x^i_t=x} \ID{u^i_t=u},
\end{equation}
where the above equation consists of $(n-1)d^{-i}_t(x)$  binary random variables with success probability $\gamma^{-i}_t(x)(u)$. In addition, the following holds:
\begin{equation}
\Exp{(n-1) \mathfrak{D}^{-i}_t(x,u)}
=(n-1)d^{-i}_t(x)\gamma^{-i}_t(x)(u),
\end{equation}
where $(n-1) \mathfrak{D}^{-i}_t(x,u)$ has Binomial probability distribution.

\section{Proof of Lemma~\ref{lemma:conditional_prob}}\label{sec:proof_lemma:conditional_prob}
The proof follows from  the fact that the conditional probability~\eqref{eq:conditional_prob_non} is invariant  to the  permutation of  the other players.  More precisely,  the primitive random variables,  dynamics and control laws of the other  players are exchangeable, which   means that all permutations are equally likely to happen. Therefore, the conditional probability~\eqref{eq:conditional_prob_non} is exchangeable, i.e.,  for any $j,k  \in \mathbb{N}_n$, $j,k \neq i$:
$\Prob{\sigma_{j,k} \mathbf x^{-i}_t \mid  x^i_{1:t},  d^{-i}_{1:t},\gamma^{i}_{1:t}, \gamma^{-i}_{1:t}}
=\Prob{\mathbf x^{-i}_t \mid  x^i_{1:t},  d^{-i}_{1:t}, \gamma^{i}_{1:t},\gamma^{-i}_{1:t}}$. 
Hence,  the above conditional probability is  representable by $d^{-i}_t$, i.e.,  it  is equal to  zero if the empirical distribution of $\mathbf x^{-i}_t$ is not $d^{-i}_t$, and   it is equal to   $1/H(d^{-i}_t)$ otherwise, where $H (d^{-i}_t)$ is the number of all  realizations of $\mathbf x^{-i}_t$ whose empirical distribution is $d^{-i}_t$. Note that function $H$ is independent of  the strategies  of players. 

\section{Proof of Lemma~\ref{lemma:cost}}\label{sec:proof_lemma:cost}
The proof follows from Lemma~\ref{lemma:conditional_prob} such that:
\begin{align}
&\Exp{c_t(x^i_t,u^i_t, \mathbf x_t, \mathbf u_t) \mid  x^i_{1:t}, d_{1:t},\gamma^i_{1:t},  \gamma^{-i}_{1:t}}\\
& =\sum_{\mathbf x^{-i}_t } \sum_{\mathbf u^{-i}_t} \big[\sum_{u \in \mathcal{U}}c_t(x^i_t,u, \mathbf x_t, \mathbf u_t) \gamma^{i}_t(x^i_t)(u)\big]  \\
&\quad \times \Prob{\mathbf x_t^{-i}, \mathbf u_t^{-i}\mid  x^i_{1:t}, d^{-i}_{1:t},\gamma^i_{1:t}, \gamma^{-i}_{1:t}}\\
&=\sum_{\mathbf x^{-i}_t } \sum_{\mathbf u^{-i}_t}  \big[\sum_{u \in \mathcal{U}}c_t(x^i_t,u, \mathbf x_t, \mathbf u_t) \gamma^{i}_t(x^i_t)(u)\big] \\
&\times  \Prob{\mathbf x_t^{-i}\mid  x^i_{1:t},  d^{-i}_{1:t},u^i_{1:t}, \gamma^{-i}_{1:t}}   \prod_{j\neq i}^n \gamma^{-i}_t(x^j_t)(u^j_t)\\
&\substack{(a)\\=:}\ell_t(x^i_{t}, d_t,\gamma^i_t,\gamma^{-i}_t),
\end{align}
where $(a)$ follows from equation~\eqref{eq:def-MF-n-1}.

\section{Proof of Lemma~\ref{lemma:evolution-lip}}\label{sec:proof_lemma:evolution-lip}
The proof follows from the triangle inequity and the fact that empirical distribution converges to the expectation at the rate $\mathcal{O}(\frac{1}{\sqrt{n}})$ in the mean-square sense. In particular,
\begin{align}\label{eq:proof_lemma_evolution}
\Exp{\Ninf{d_{t+1} - m_{t+1}} } \leq  \mathbb{E}\Ninf{\Exp{d_{t+1}}- m_{t+1}} + \mathbb{E}\Ninf{d_{t+1} - \Exp{d_{t+1}}} \nonumber \\
\substack{(a)\\=} \mathbb{E}\Ninf{\hat f_t(d_t,\gamma_t) - \hat f_t(m_t,\gamma_t)} + \mathbb{E}\Ninf{d_{t+1} - \Exp{d_{t+1}}} \nonumber \\
\substack{(b)\\ \leq}  K^m_t \Ninf{d_t -m_t} + \mathcal{O}(\frac{1}{\sqrt{n}}),
\end{align}
where $(a)$ follows from the the fact that for any $y \in \mathcal{X}$,
\begin{align}
\Exp{d_{t+1}(y)|d_t,\gamma_t}=\Exp{\sum_{i=1}^n \ID{x^i_{t+1}=y}\mid d_t,\gamma_t}\\
=\sum_{x \in \mathcal{X}} d_t(x) \Prob{x^i_{t+1}=y |x^i_t=x, \gamma^i_t(x)=\gamma_t(x),d_t}\\
=\sum_{x \in \mathcal{X}} d_t(x)\mathcal{T}_t(y,x,\gamma_t(x),d_t)=\hat f_t(d_t,\gamma_t))(y),
\end{align}
and $(b)$ follows from~\cite[Lemmas 1 and 2]{JalalCDC2017}. In addition, for any $x \in \mathcal{X}$ and $u \in \mathcal{U}$,
\begin{align}
\mathbb{E}|\mathfrak{D}_{t+1}(x,u)- \mathfrak{M}_{t+1}(x,u)|\leq \mathbb{E}| \Exp{\mathfrak{D}_{t+1}(x,u)}-\mathfrak{M}_{t+1}(x,u)| \\
+\mathbb{E}|\mathfrak{D}_{t+1}(x,u)-\Exp{\mathfrak{D}_{t+1}(x,u)}|\\
\substack{(c)\\=} \Exp{d_{t+1}(x) \gamma(x)(u) - m_{t+1}(x) \gamma(x)(u) } + \mathcal{O}(\frac{1}{\sqrt n})\\
\substack{(d)\\ \leq }  K^m_t \Ninf{d_t -m_t} + \mathcal{O}(\frac{1}{\sqrt{n}}),
\end{align}
where $(c)$ follows from Lemma~\ref{lemma:joint_empirical} and $(d)$ follows from~\eqref{eq:proof_lemma_evolution}. 

\section{Proof of Lemma~\ref{lemma:Lipschitz}}\label{sec:proof_lemma:Lipschitz}
For ease of  display, let  $\tilde{\mathfrak{D}}_{t}:=\frac{n}{n-1} \mathfrak{D}_{t} + \frac{1}{n-1} \delta(x_{t})$, $t \in \mathbb{N}_T$. The proof follows from a backward induction and equation~\eqref{eq:DP_DSS} such that  at the terminal time $t=T$ and  any state  $x \in \mathcal{X}$, one has:
\begin{align}
&V_T(x, d_T)=\min_{\gamma_T(x) } \ell_T(x,d_T,\gamma_T(x),\gamma_T)\\
&=\min_{\gamma_T(x) } \Exp{c_T(x,u_T,\mathfrak{D}_T)\mid x,d_T,\gamma_T(x), \gamma_T}\\
&=\min_{\gamma_T(x)}\mathbb{E}_{\tilde{\mathfrak{D}_{T}}} [\sum_{u \in \mathcal{U}}c_T(x,u, \frac{n-1}{n}\tilde{\mathfrak{D}}_T+\frac{1}{n}\delta(x,u)) \gamma_T(x)(u)]\\
& \substack{(a)\\ \leq } \min_{\gamma_T(x)}\mathbb{E}_{\tilde{\mathfrak{D}_{T}}}[ \sum_{u \in \mathcal{U}}c_T(x,u, \frac{n-1}{n}\mathfrak{D}^{-i}_T+\frac{1}{n}\delta(x,u)) \gamma_T(x)(u)\\
&\quad -\sum_{u \in \mathcal{U}}c_T(x,u,\mathfrak{M}_T) \gamma_T(x)(u)]\\
&+\mathbb{E}_{\tilde{\mathfrak{D}_{T}}}[ \sum_{u \in \mathcal{U}}c_T(x,u,\mathfrak{M}_T) \gamma_T(x)(u) ]\\
& \substack{(b) \\ \leq } \Exp{ \|\frac{n-1}{n}\mathfrak{D}^{-i}_T-\mathfrak{M}_T \|}+\frac{1}{n}+ \min_{\gamma_T(x)} \sum_{u \in \mathcal{U}}c_T(x,u,\mathfrak{M}_T) \gamma_T(x)(u)\\
&\substack{(c)\\=} K^c_T \Ninf{d_T -m_T} + \mathcal{O}(\frac{1}{\sqrt n}) + \hat V_t(x_T,m_T),
\end{align}
where $(a)$ follows from the triangle inequality,  per-step cost being non-negative (by definition), and the monotonicity of the minimum operator; $(b)$ follows from Assumption~\ref{assumption: finite-lipschitz}, equation~\eqref{eq:cost_infinite} and the monotonicity of the minimum operator, and $(c)$ follows from~\eqref{eq:hat-V}.  Assume now that   inequality~\eqref{eq:inequality-value}  holds at time $t+1$ for any $x \in \mathcal{X}$, i.e.,  
\begin{multline}\label{eq:indunction-t+1-approximate}
| V_{t+1}(x,d_{t+1}) -\hat{V}_{t+1}(x,m_{t+1}) | \leq K^v_{t+1} \| d_{t+1} -m_{t+1}\| \\ + K^o_{t+1}\mathcal{O}(\frac{1}{\sqrt{n}}).
\end{multline}
The objective is to show   that  it holds  at time $t$ as well.  It follows from \eqref{eq:DP_DSS} that for any $x_t \in \mathcal{X},$
\begin{align}\label{eq:proof_ndunction-t-approximate}
V_t(x_t,d_t)&=\min_{\gamma_t(x_t)} \left( \ell_t(x_t,d_t,\gamma_t(x_t),\gamma_t)+\Exp{ V_{t+1}(x_{t+1},d_{t+1})} \right) \nonumber \\
&=\min_{\gamma_t(x_t)} (\ell_t(x_t,d_t,\gamma_t(x_t),\gamma_t) \pm \ell_t(x_t,m_t,\gamma_t(x_t),\gamma_t) \nonumber \\
&+\mathbb{E}_{\tilde{\mathfrak{D}}_{t+1}} [\sum_{x_{t+1}} A_t(C_t-D_t) +(A_t-B_t )D_t + B_t D_t]),
\end{align}
where
\begin{align}
A_t&:=\Prob{x_{t+1} \mid x_t,d_t,\gamma_t(x_t)},\\
B_t&:=\Prob{x_{t+1} \mid x_t, m_t, \gamma_t(x_t)},\\
C_t&:=V_{t+1}(x_{t+1},d_{t+1}),\\
D_t&:=\hat V_{t+1}(x_{t+1}, m_{t+1}).
\end{align}
We now  find an upper bound for  each  term in~\eqref{eq:proof_ndunction-t-approximate}. From Assumption~\ref{assumption: finite-lipschitz} and equation~\eqref{eq:cost_infinite}, it  results that:
\begin{align}
&\Ninf{\ell_t(x_t,d_t,\gamma_t(x_t),\gamma_t) - \ell_t(x_t,m_t,\gamma_t(x_t),\gamma_t} \\
& \mathbb{E}_{\tilde{\mathfrak{D}_t}} [\sum_{u \in \mathcal{U}} c_t(x_t,u,\frac{n-1}{n}\tilde{\mathfrak{D}}_{t}+\frac{1}{n} \delta(x_t,u))\gamma_t(x_t)(u)]\\
&-\mathbb{E}_{\tilde{\mathfrak{D}_t}} [\sum_{u \in \mathcal{U}} c_t(x_t,u, \mathfrak{M}_t)\gamma_t(x_t)(u)]\\
&\leq K^c_t \Ninf{d_t - m_t} +\mathcal{O}(\frac{1}{\sqrt n}).
\end{align}
In addition, from Lemma~\ref{lemma:evolution-lip} and equation~\eqref{eq:indunction-t+1-approximate}, one arrives at:
\begin{align}
\sum_{x_{t+1}} A_t(C_t-D_t)  \leq K^v_{t+1} \Ninf{d_{t+1} - m_{t+1}} + K^o_{t+1} \mathcal{O}(\frac{1}{\sqrt n})\\
\leq K^v_{t+1} K^m_t \Ninf{d_t - m_t} + (K^v_{t+1}+  K^o_{t+1}) \mathcal{O}(\frac{1}{\sqrt n}).
\end{align}
From Assumption~\ref{assumption: finite-lipschitz} and Lemma~\ref{lemma:uniform_bounded_cost}, it follows that:
\begin{align}
\sum_{x_{t+1}} (A_t-B_t) D_t  \leq K^p_t (\sum_{\tau=1}^{t+1} \beta^{\tau}K^c_\tau)  \Ninf{d_t - m_t}. 
\end{align}
Note that $
\sum_{x_{t+1}}  B_t D_t= \Exp{ \hat V_{t+1}(x_{t+1}, m_{t+1}) \mid x_t, m_t, \gamma_t}$.
Therefore, it results from~\eqref{eq:proof_ndunction-t-approximate} that
\begin{align}
&V_t(x_t,d_t) \leq K^v_t \Ninf{d_t - m_t} + K^o_{t} \mathcal{O}(\frac{1}{\sqrt n}) \\
&+ \min_{\gamma_t(x_t)} \ell_t(x_t,m_t,\gamma_t(x_t),\gamma_t)+ \Exp{ \hat V_{t+1}(x_{t+1}, m_{t+1}) \mid x_t, m_t, \gamma_t}\\
&= K^v_t \Ninf{d_t - m_t} + K^o_{t} \mathcal{O}(\frac{1}{\sqrt n})+ \hat V_t(x_t,m_t).
\end{align}

\section{Proof of Theorem~\ref{thm:finite-ns-convergence}}\label{sec:proof_thm:finite-ns-convergence}
From the triangle inequality, it results that
\begin{equation}\label{eq:triangle}
|J^\ast_n - \hat J_n | \leq |J^\ast_n -\Exp{\hat V_1(x_1,m_1)}|+ |\hat J_n -\Exp{\hat V_1(x_1,m_1)}|.
\end{equation}
The first term of the right-hand side of~\eqref{eq:triangle} is bounded as follows, on noting that  $J^\ast_n=\Exp{V_1(x_1,d_1)}$. For any time $t \in \mathbb{N}_T$ and state $x_1 \in \mathcal{X}$,
\begin{align}
&|\mathbb{E} [V_1(x_1,d_1)] - \mathbb{E }[\hat V_1(x_1,m_1)]|  \substack{(a) \\ \leq } \mathbb{E} |V_1(x_1,d_1) - \hat V_1(x_1,m_1) | \nonumber  \\
& \substack{(b) \\ \leq } K^v_1 \Exp{\Ninf{d_1 -m_1} }+ K^o_1 \mathcal{O}(\frac{1}{\sqrt n})\\
& \substack{(c) \\ \leq }  K^v_1 \mathcal{O}(\frac{1}{\sqrt n})+ K^o_1 \mathcal{O}(\frac{1}{\sqrt n}),
\end{align}
where $(a)$ follows from the monotonicity of the expectation operator;  $(b)$ follows from  Lemma~\ref{lemma:Lipschitz},  and $(c)$ follows from  the relation $m_1=P_X$ and  the fact  that the empirical distribution $d_1$ converges to  its limit $P_X$ in the mean-square sense at the rate $\mathcal{O}(\frac{1}{n})$ (see~\cite[Lemma~2]{JalalCDC2017} for more details). Similarly,  the second term of the right-hand side of~\eqref{eq:triangle} is also bounded by $\mathcal{O}(\frac{1}{\sqrt n})$, on noting that initially $\hat x_1=x_1$ and $\hat d_1=d_1$.  

\end{document}